\RequirePackage{fix-cm}
\documentclass[smallextended,envcountsect,envcountsame]{svjour3hack}       
\smartqed  
\usepackage{graphicx}
\usepackage{amsmath, amssymb, enumitem,xcolor}
\usepackage[utf8]{inputenc}
\usepackage{xurl}
\usepackage{float}
\usepackage{multirow}

\newcommand{\rank}{\operatorname{rank}}

\makeatletter
\renewcommand*{\top}{%
  {\mathpalette\@transpose{}}%
}
\newcommand*{\@transpose}[2]{%
  \raisebox{\depth}{$\m@th#1\mathsf{T}$}%
}
\makeatother

\usepackage[hypertexnames=false,colorlinks=true,breaklinks=true,bookmarks=true,urlcolor=blue,citecolor=blue,linkcolor=blue,bookmarksopen=false,draft=false]{hyperref}

\newtheorem{cor}{Corollary}[section]

\makeatletter
\let\c@prop\c@theorem
\let\c@cor\c@theorem
\let\c@lemma\c@theorem
\let\c@definition\c@theorem
\let\c@example\c@theorem
\let\c@remark\c@theorem
\let\c@obs\c@theorem
\let\c@claim\c@theorem
\makeatother

\newcommand{\R}{\mathbb{R}}
\newcommand{\Z}{\mathbb{Z}}
\newcommand{\bA}{\mathbf{A}}
\newcommand{\bB}{\mathbf{B}}
\newcommand{\bQ}{\mathbf{Q}}
\newcommand{\bk}{\mathbf{k}}
\newcommand{\ba}{\mathbf{a}}
\newcommand{\bb}{\mathbf{b}}
\newcommand{\bc}{\mathbf{c}}
\newcommand{\be}{\mathbf{e}}
\newcommand{\bv}{\mathbf{v}}
\newcommand{\bt}{\mathbf{t}}
\newcommand{\bw}{\mathbf{w}}
\newcommand{\bx}{\mathbf{x}}
\newcommand{\by}{\mathbf{y}}
\newcommand{\bz}{\mathbf{z}}
\DeclareMathOperator{\vol}{vol}

\DeclareMathOperator{\cl}{cl}

\DeclareMathOperator{\conv}{conv}

\begin{document}

\title{Gaining or Losing Perspective for Convex Multivariate Functions  on a Simplex\thanks{
This work was supported in part by ONR grant N00014-21-1-2135. 
}
}

\titlerunning{Gaining or Losing Perspective for Convex Multivariate Functions on a Simplex}        

\author{Luze Xu         \and
        Jon Lee 
}


\institute{L. Xu, J. Lee: \at
              University of Michigan,
              \email{$\{$xuluze,jonxlee$\}$@umich.edu}           
}

\date{\today}

\maketitle

\begin{abstract}
MINLO (mixed-integer nonlinear optimization) formulations of the disjunction between the origin and a polytope via a binary indicator variable have broad applicability  in nonlinear combinatorial optimization, for modeling a fixed cost $c$ associated with 
carrying out a set of $d$ activities and a convex variable cost function $f$ associated with the levels of the activities. 
The perspective relaxation
is often used to solve such models to optimality in a branch-and-bound context, especially in the 
context in which $f$ is univariate (e.g., in Markowitz-style portfolio optimization). 
But such a relaxation typically requires conic solvers
and are typically not compatible 
with general-purpose NLP software which can accommodate additional classes of constraints.
This motivates the study of weaker relaxations to investigate when simpler relaxations may be adequate.
Comparing the volume (i.e., Lebesgue measure)
of the relaxations as means of comparing them, we lift some of the results related to univariate functions $f$ to the multivariate case. Along the way, we survey, connect and extend 
relevant results on integration over a simplex, some of which we concretely employ, and others of which can be used for further exploration on our main subject.

\end{abstract}

\section{Introduction}
The ``perspective reformulation'' technique is used to obtain strong relaxations of the MINLO (mixed-integer nonlinear optimization) formulations modeling indicator variables: when an indicator variable is ``off'', a vector of $d$ decision variables is forced to some specific point (often $\mathbf{0}\in\mathbb{R}^d$), and when it is ``on'', the vector of decision variables must belong to a specific convex set in $\mathbb{R}^d$ (see \cite{gunlind1,lee_gaining_2020} and the many references therein). 

Perspective relaxations typically contain conic constraints, but not all NLP solvers are equipped to handle conic constraints correctly. Conic solvers (like \verb;MOSEK; and \verb;SDPT3;; see \cite{mosek} and \cite{SDPT3}, respectively) handle
such constraints coming from well-known classes of cones (e.g., second-order cones, power cones, exponential cones), by providing associated barrier functions. But they do not have the capability to handle 
all such constraints. Even in cases where a conic solver can handle the perspectivization of
a given convex function, there may be other (even convex) constraints that such a solver cannot handle.
In such a situation, we may hope to use a general NLP solver, which we might also expect to be faster than a
conic solver, but these are not typically able to handle
perspective functions correctly (see \cite[ Sec. 1.2]{lee_gaining_2020}). 

For the univariate case of a continuous variable $x$ being either $0$ or in a positive interval $[\ell,u]$,
\cite{lee_gaining_2020,lee2020piecewise} studied the trade-off between the tightness and tractability of alternative relaxations, and proposed several natural and simpler non-conic-programming relaxations.
For the specific case of $f(x):=x^p$, $p>1$, they obtained concrete results, considering the 
relative tightness of formulations as  functions of $\ell$, $u$, and $p$. 
These results apply to the  situation where indicator variables manage terms in a \emph{separable} objective function, with each continuous variable being either $0$ or in an interval (not containing $0$).
%

In what follows, we consider the multivariate case in 
which the decision variable (vector) $\bx$
is either $\mathbf{0}\in\mathbb{R}^d$ or in a simplex $J\subset \mathbb{R}_{\geq 0}^d$
(not containing $\mathbf{0}$).
Our goal is to lift results related to univariate functions from  
\cite{lee_gaining_2020,lee2020piecewise} to the multivariate case.
The idea of comparing relaxations via their volumes (i.e., Lebesgue measure)
was introduced in \cite{LM1994} (also see \cite{LeeSkipperSpeakmanMPB2018}, and the many references therein).
\cite{lee_gaining_2020,lee2020piecewise} first developed these ideas in the 
context of perspective relaxation, for the univariate case.
Following \cite{lee_gaining_2020,lee2020piecewise}, we also use $(d+2)$-dimensional volume as a measure for comparing relaxations. We have $\bx\in \mathbb{R}_{\geq 0}^d$, a binary indicator variable $z$ 
keeping track of whether $\bx=\mathbf{0}$ or $\bx\in J$, and a further variable $y\in \mathbb{R}$
which ``captures'' $f(\bx)$; so $d+2$ variables in total. 

\medskip
\noindent {\bf Organization and contributions.} In what follows, we formally define our sets of interest:
a disjunctive set, the perspective relaxation, and the na\"{i}ve relaxation. 
In Section \ref{sec:PR_vol}, we derive general formulae for the volumes of the perspective relaxation and na\"{i}ve relaxation. These formulae both require integrating over a simplex, and so 
in Section \ref{sec:int_simplex}, we survey, connect and extend relevant results in the literature concerning integration over a simplex.
This survey is an important contribution of our work as
it collects these results for optimizers, in one place and in an accessible form;
we rely on some of these results in Section \ref{sec:multi}, and other results 
are natural tools that could be used to push forward further on our motivating topic.
In Section \ref{sec:multi}, we derive the formula for the volume of the na\"{i}ve relaxation for two natural families of functions, generalizing what is known for the univariate case.
We also demonstrate how to work numerically, when there is no closed-form integration
formula. In Section \ref{sec:conc}, we make some brief conclusions.

\medskip
\noindent {\bf Notation.} In what follows, we use boldface lower-case for vectors and boldface upper-case for matrices. $\displaystyle\be_n^{d}$ denotes the $n$-th unit vector in $\R^{d}$, and the superscript $d$ is often dropped if the dimension is clear from the context. 
$\Delta_d:=\{\bx\in\R^d_{\geq 0} ~:~ \sum_{j=1}^d x_j\le 1\}=\conv\{\mathbf{0},\be_1,\be_2,\dots,\be_d\}$ denotes the standard $d$-simplex in $\R^d$.
$J:=\conv\{\bv_0,\bv_1,\dots,\bv_d\}\subset \R_{\geq 0}^d$ denotes an arbitrary $d$-simplex in $\R_{\geq 0}^d$, where $\bv_0,\bv_1,\dots,\bv_d$ are the $d+1$ (affinely independent) vertices of $J$. 
An affine transformation from $J$ to $\Delta_d$ can be used to extend integration results from $\Delta_d$ to a general $J$ (see for example, \cite{lasserre_simple_2020,rouigueb_integration_2019}): 
\begin{equation}\label{eqn:affine_transformation_B}
\bx\in J \quad \Leftrightarrow \quad \bB^{-1}(\bx-\bv_0)\in\Delta_d\thinspace,
\end{equation}
where $\bB: = \left[\bv_1-\bv_0,~\dots~,\bv_d-\bv_0\right]$.
Finally, we use $i$ to denote the imaginary unit.

\medskip

Generally, we can triangulate any polytope into simplices (see for example, \cite{de2010triangulations}). So we consider the case when the convex set is a simplex as both a natural and fundamental starting point and a building block for the case of a general polytope. 
Thus, we concentrate our efforts on considering models for the case of a single simplex.
That is, in what follows, we have a convex function $f:~\R^d_{\geq 0}\rightarrow \R$
and a single simplex $J \subset \R^d_{\geq 0}\setminus \{\mathbf{0}\}$. 
%
We define the disjunctive set
\begin{align*}
& D(f,J) := \left\{\mathbf{0}_{d+2}\right\}\cup D_1(f,J), \mbox{ where} \\
& D_1(f,J) : = \left\{(\bx,y,1)\in \R^d\times \R \times \{0,1\} ~:~
\mu(\bx) \geq y \geq f(\bx), ~\bx\in J\right\},
\end{align*}
and $\mu(\bx)$ is the linear function that agrees with $f(\bx) $ at the $d+1$ vertices of the simplex $J$.
Note that because we assume that $f$ is convex, we have that $f(\bx) \leq \mu(\bx)$ for $\bx\in J$.
Within the hyperplane determined by $z=1$, the set $D_1(f,J)$ is the convex hull of the graph
of $f$ on the domain $J$. The disjunction models the choice of either $\bx =\mathbf{0}_d$~, $y=0$
or $(\bx,y)$ is in the convex hull of the graph of $f$ on the domain $J$. The disjunction does this
through the variable $z\in\{0,1\}$.

The \emph{perspective function} 
\[
\tilde{f}(\bx,z):= \left\{
          \begin{array}{ll}
             z f(\bx/z), & \hbox{for $z>0$;} \\
             +\infty, & \hbox{otherwise}
          \end{array}
         \right.
\]
is well-known to be a convex function, when $f:~\R^d \rightarrow \R$ is convex  (see \cite[Sec. IV.2.2, Page 160]{perspecbook}). 
Importantly, if we evaluate the closure of $\tilde{f}$ (whose epigraph is the closure of the epigraph of $\tilde{f}$, see \cite[Sec. IV, Definition 1.2.4, Page 149]{perspecbook})
at $(\mathbf{0}_d,0)$, we get $0$ (see \cite[Sec. IV, Remark 2.2.3, Page 162]{perspecbook}).
So we can define the \emph {(higher-dimensional) perspective relaxation}
\begin{align*}
\hypertarget{Pfj}{P(f,J)}
:=&\cl \left\{ (\bx,y,z) \in \R^d\times \R \times \R ~:~ \right.\\
&\left. \qquad \tilde{\mu}(\bx,z) \geq y \geq \tilde{f}(\bx,z), ~\bx\in z\cdot J,~ 1\geq z > 0
\vphantom{\R^d}\right\},
\end{align*}
where the upper bound $\tilde{\mu}(\bx,z)$ is the perspective function of the linear function $\mu(\bx)$, and is thus linear itself.

Some comments on \hyperlink{Pfj}{\hyperlink{Pfj}{$P(f,J)$}}:
\begin{itemize}
\item \hyperlink{Pfj}{\hyperlink{Pfj}{$P(f,J)$}} intersects the hyperplane defined by $z=0$ at the single point $\mathbf{0}_{d+2}$, and it intersects the hyperplane defined by $z=1$ at $D_1(f,J)$. It is clear that \hyperlink{Pfj}{\hyperlink{Pfj}{$P(f,J)$}} is the convex hull of $D(f,J)$.
\item When $d=1$, \hyperlink{Pfj}{\hyperlink{Pfj}{$P(f,J)$}} is the perspective relaxation of  \cite{lee_gaining_2020} (and others). 
\item If the simplex $J\subset \R^d_{\geq 0}$ is described by linear inequalities $\bA\bx\le \bb$, then we can write $\bx\in z\cdot J$ as the homogeneous system $\bA\bx\le \bb z$. 
\item The constraint $y \geq f(\bx)$ is equivalent to $(\bx,y,1)\in K_f$, where 
\[
K_f:= \left\{
(\bx,y,z) \in \R^d\times \R \times \R ~:~  y \geq \tilde{f}(\bx,z),~ z>0
\right\}
\]
is a convex cone, whose closure is \hyperlink{Pfj}{\hyperlink{Pfj}{$P(f,J)$}} without the upper bound $\tilde{\mu}(\bx,z)$. So relaxing the disjunction $D(f,J) $ to \hyperlink{Pfj}{\hyperlink{Pfj}{$P(f,J)$}} enables us to use an interior-point conic solvers (like \verb;MOSEK; and \verb;SDPT3;)
\emph{whenever appropriate barriers are available in the solver}.  
\item Considering $f(\bx)$ at each of the $d+1$ vertices of $J$, there is a unique hyperplane in the variables $(\bx,y)\in\R^{d+1}$ passing through these $d+1$ points.
Suppose that $J:=\conv\{\bv_0,\bv_1,\dots,\bv_d\}\subset \R^d$. 
Then $\mu(x)$ can be defined as
\begin{equation}\label{eqn:mu}
    \mu(x): = \bw^\top
\bB^{-1}(\bx-\bv_0) + f(\bv_0),
\end{equation}
where  $\bB$ is from \eqref{eqn:affine_transformation_B}, and $\bw^\top:=\left[f(\bv_1)-f(\bv_0),~\dots~,f(\bv_d)-f(\bv_0)\right]\in\R^{1\times d}$. 
Therefore, 
\[
\tilde{\mu}(\bx,z) = z\mu(\bx/z) = 
\bw^\top \bB^{-1}(\bx-z \bv_0) + f(\bv_0)z~.
\]
\end{itemize}

Extending a key setting from   \cite{lee_gaining_2020}, 
we consider the following special case: the domain of the convex function $f$ is $\conv(J\cup \{\mathbf{0}\})=\{z\cdot J: 0\le z\le 1\}$, and $f(\mathbf{0})=0$. We can then define the \emph{(higher-dimensional) na\"{i}ve relaxation}:
\begin{align*}
\hypertarget{P0fj}{P^0(f,J)}
:=&\left\{ (\bx,y,z) \in \R^d\times \R \times \R ~:~\right.\\
&\qquad\left.\tilde{\mu}(\bx,z) \geq y \geq f(\bx), ~\bx\in z\cdot J,~ 1\geq z \geq 0 \right\},
\end{align*}
where the upper bound $\tilde{\mu}(\bx,z)$ is defined as in the perspective function of $\mu(\bx)$.
It is clear that \hyperlink{P0fj}{$P^0(f,J)$} is convex, due to the convexity of $f$ and the linearity of the other constraints.

Given any $\bx,z$ such that $\bx\in z\cdot J$ and $1\geq z\geq 0$, we have $z f(\bx/z) + (1-z) f(\mathbf{0}) \ge f(\bx)$ because of the convexity of $f$, which implies that $z f(\bx/z) \ge f(\bx)$. From this, we can see that $\hyperlink{Pfj}{P(f,J)}  \subseteq \hyperlink{P0fj}{P^0(f,J)} $, i.e., the na\"{i}ve relaxation contains the perspective relaxation (as holds when $d=1$),  which implies that \hyperlink{P0fj}{$P^0(f,J)$} is also a relaxation of $D(f,J)$.
We can readily see that \hyperlink{P0fj}{$P^0(f,J)$} is easier to handle than \hyperlink{Pfj}{$P(f,J)$},
because it involves $f$ rather than the perspective function of $f$. 
So it is natural to try and understand,  depending on $f$ and $J$,
how much stronger 
\hyperlink{Pfj}{$P(f,J)$} is compared to \hyperlink{P0fj}{$P^0(f,J)$}.

Figure \ref{fig:NRPR} shows the comparison of the  na\"{i}ve and perspective relaxation for an example in the univariate case. The left subfigure shows $D(f,J)=\{\mathbf{0}_{d+2}\}\cup D_1(f,J)$ (in black) and the lower bounds $f(\bx)$ (in blue) and $\tilde{f}(\bx,z)$ (in yellow); the right subfigure is a cross section on the plane $z=\frac12$ for the upper bound $\tilde{\mu}(\bx,z)$ (in black) and two lower bounds in the relaxations.
\begin{figure}[H]
    \centering
    \includegraphics[width=0.5\textwidth]{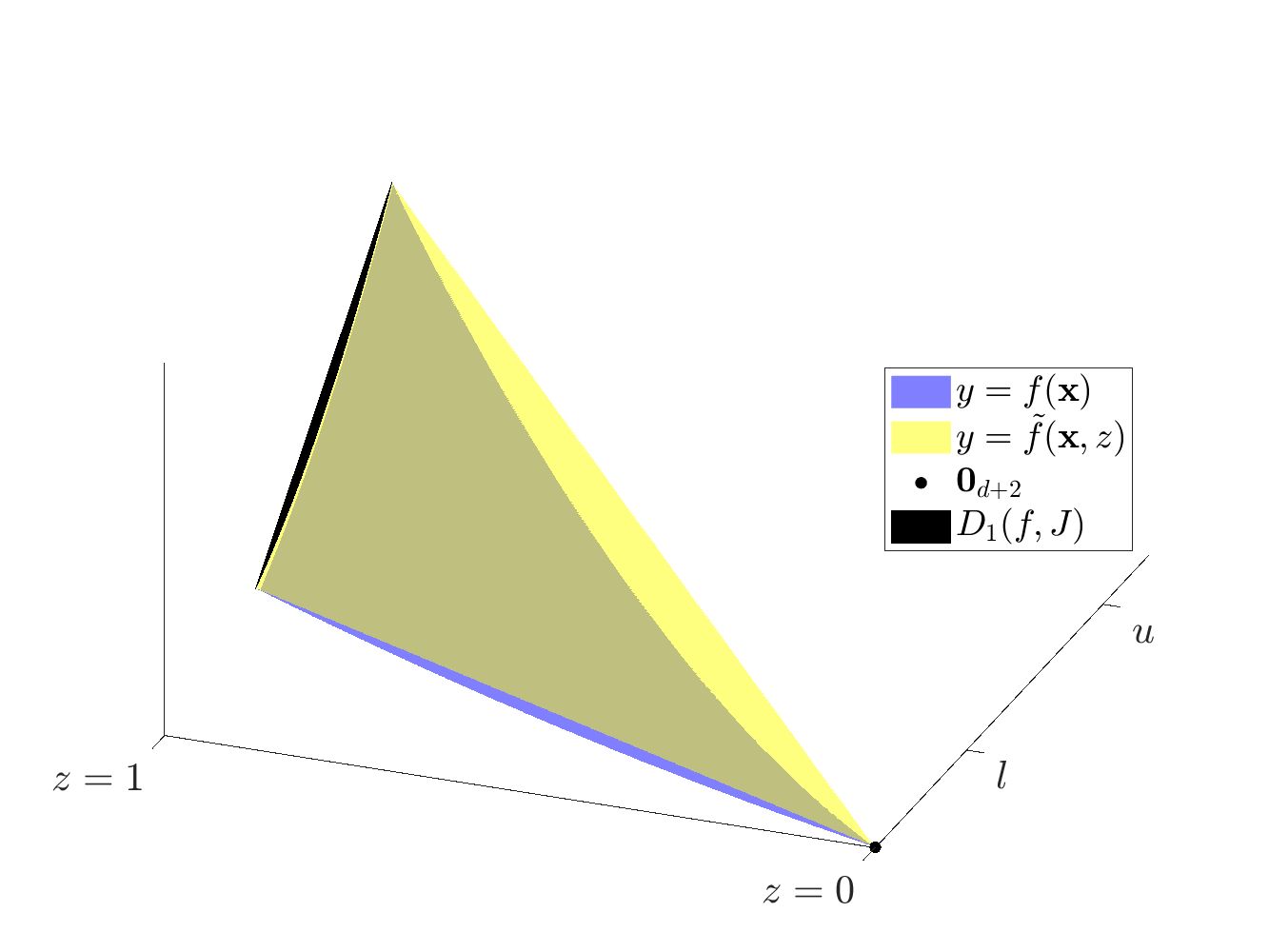}
    \includegraphics[width=0.45\textwidth]{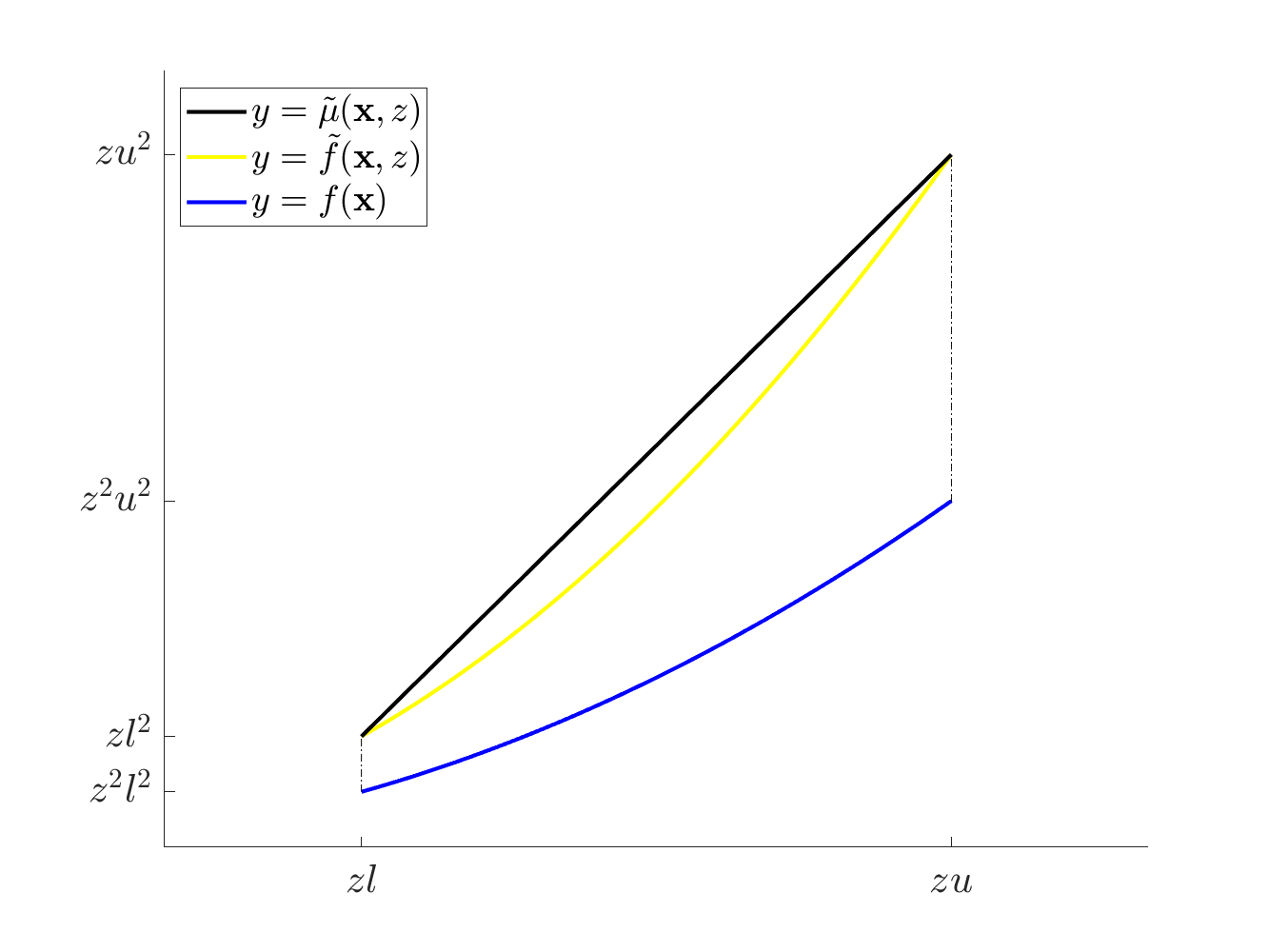}
    \caption{The na\"{i}ve and perspective relaxation for $d=1$, $f(\bx) = x_1^2$, $J=[\ell,u]=[2,5]$.}
    \label{fig:NRPR}
\end{figure}

The univariate case can be used to handle \emph{separable} multivariate convex functions. In what follows, we are more interested in \emph{nonseperable} convex functions, like $(\bc^\top \bx)^n$, $\mathrm{e}^{\bc^\top\bx}$, and the following example.

\begin{example}\label{ex:logsumexp}
A nice nonseparable convex function is the ``log-sum-exp''\footnote{ We subtract a constant $\log d$ from 
the ususal ``log-sum-exp'' function $\log\sum_{j=1}^d\mathrm{e}^{x_j}$ to satisfy $f(\mathbf{0})=0$. See \url{https://docs.scipy.org/doc/scipy/reference/generated/scipy.special.logsumexp.html}} function
$f(\bx):=\log  \frac{1}{d}\sum_{j=1}^d\mathrm{e}^{x_j}$, which is a smooth  under-estimator of the function $\max\{x_1,\dots,x_d\}$ $ \left(f(\bx)\le \max\{x_1,\dots,x_d\},\lim_{u\rightarrow \infty} \frac{f(u\bx)}{u}=\max\{x_1,\dots,x_d\}\right)$.
The ``log-sum-exp'' inequality  $y\geq f(\bx)$ \emph{could} be modeled with exponential-cone constraints (see \cite[Sec. 5.2.6]{mosek} and \cite{mosekcheat}):
\begin{align*}
& \sum_{j=1}^ d w_j \leq d,\\
    & (w_j,1,x_j-y) \in K_{\exp}:=\left\{ (u_1,u_2,u_3):~ u_1 \geq u_2 \mathrm{e}^{u_3/u_2},~ u_2\geq 0 \right\},\\
   &\qquad \mbox{ for } j=1,\ldots,d.
\end{align*}
\end{example}
But in fact, even general nonlinear-programming solvers can comfortably work directly with  $y\geq \log \frac1d\sum_{j=1}^d  \mathrm{e}^{x_j}$. So, if we are satisfied with the associated na\"{i}ve relaxation (notice that $f(\mathbf{0})=0$), we can reliably use
general nonlinear-programming solvers. 

Going further, perspectivizing, we are led to the stronger inequality  
 $y\geq z \log \frac1d\sum_{j=1}^d  \mathrm{e}^{x_j/z}$, which is not well handled by
 general nonlinear-programming solvers. But, similarly to how we conically modeled
 $y\geq \log \frac1d\sum_{j=1}^d  \mathrm{e}^{x_j}$, 
 we can model $y\geq z \log \frac1d\sum_{j=1}^d  \mathrm{e}^{x_j/z}$
  by 
 \begin{align*}
& \sum_{j=1}^ d w_j \leq dz,\\
    & (w_j,z,x_j-y) \in K_{\exp}:=\left\{ (u_1,u_2,u_3):~ u_1 \geq u_2 \mathrm{e}^{u_3/u_2},~ u_2\geq 0 \right\},\\
   &\qquad \mbox{ for } j=1,\ldots,d,
\end{align*}
which is nicely handled in \verb;MOSEK;, \emph{but which now restricts what other types of (even nonconvex) constraints can be in a model and could generally lead to slower solves.} 

\section{Volumes of relaxations}\label{sec:PR_vol}

The well-known volume formula for the $d$-simplex $J:=\conv\{\bv_0,\bv_1,\dots,\bv_d\}\subset \R^d$ is 
$$
\vol(J) = \int_{J}1 d\bx = \frac{1}{d!}\left|\det\left[\bv_1-\bv_0,~\dots~,\bv_d-\bv_0\right]\right| =\frac{1}{d!}\left|\det\begin{bmatrix}
    \bv_0 & \bv_1 & \dots & \bv_d\\
    1   &  1  & \dots & 1
\end{bmatrix}\right|.
$$
\begin{lemma}\label{lem:int_mu}
    Suppose that the $d$-simplex $J:=\conv\{\bv_0,\bv_1,\dots,\bv_d\}\subset \R^d$, and $\mu(\bx)$ is defined by \eqref{eqn:mu}. Then
    $$
    \int_{J} \mu(\bx) d\bx =\frac{1}{d+1}\vol(J)\sum_{j=0}^d f(\bv_j)=\frac{1}{(d+1)!}\left|\det\begin{bmatrix}
        \bv_0 & \bv_1 & \dots & \bv_d\\
        1   &  1  & \dots & 1
    \end{bmatrix}\right|
    \sum_{j=0}^d f(\bv_j).
    $$
\end{lemma}
\begin{proof}
    \begin{align*}
        \int_{J} \mu(\bx) d\bx &= \int_{J}  \left(\bw^\top \bB^{-1}(\bx-\bv_0) + f(\bv_0)\right)~d\bx\\
        &= d!\vol(J)\int_{\Delta_d}  \left(\bw^\top \bt +f(\bv_0)\right)~d\bt~,
    \end{align*}
    %
    %
    %
    where $\Delta_d=\conv\{\mathbf{0},\be_1,\dots,\be_d\}$. We use Lemma \ref{lem:simplex} (from the next section) 
    to calculate the exact integral of a linear form over a simplex:
    $$
    \int_{\Delta_d} \bw^\top \bt ~ d\bt = \frac{1}{(d+1)!}\sum_{j=1}^{d}w_j~.
    $$
    Therefore,
    $$
    \int_{J} \mu(\bx) d\bx =\vol(J)\left(\frac{1}{d+1}\sum_{j=0}^d (f(\bv_j)-f(\bv_0)) + f(\bv_0)\right),
    $$
  and the result follows. \qed
\end{proof}

Generalizing from the univariate case \cite[Thm. 1]{lee_gaining_2020}, 
we have the following simple formula for the 
volume of \hyperlink{Pfj}{$P(f,J)$}.
\begin{theorem}\label{thm:perspecvol}
    Suppose that $f$ is a continuous and convex function 
    on the $d$-simplex $J:=\conv\{\bv_0,\allowbreak \bv_1,\dots,\bv_{d}\} \subset \R^d_{\geq 0}\setminus \{ \mathbf{0}\}$. Then
    $$
    \vol(P(f,J)) = \frac{1}{(d+2)!}\left|\det\begin{bmatrix}
        \bv_0 & \bv_1 & \dots & \bv_d\\
        1   &  1  & \dots & 1
    \end{bmatrix}\right|
    \sum_{j=0}^d f(\bv_j) - \frac{1}{d+2}\int_{J} f(\bx)d\bx .
    $$
\end{theorem}
\begin{proof}
    Notice that \hyperlink{Pfj}{$P(f,J)$} is a hyperpyramid in $\R^{d+2}$ with apex $\mathbf{0}_{d+2}$ and base a $(d+1)$-dimensional convex set in the $z=1$ hyperplane  defined by the system of inequalities,
    \begin{align*}
     &\mu(\bx) \geq y \geq f(\bx)\\
     &\bx \in J.
    \end{align*}
  The volume of such a hyperpyramid is $\frac{1}{d+2}\mathcal{B}\mathcal{H}$, where $\mathcal{B}$ is the $(d+1)$-dimensional volume of the base, and $\mathcal{H}$ is the perpendicular height of the apex over the affine span of the base.
    In this hyperpyramid, $\mathcal{H}=1$ because the apex is $\mathbf{0}_{d+2}$ and the hyperplane containing the base is defined by the
    equation $z=1$.
    We only need to compute the volume of the base via the integral
    \begin{align*}
      \mathcal{B} &= \int_{J} (\mu(\bx) - f(\bx)) d\bx\\
          &= \frac{1}{(d+1)!}\left|\det\begin{bmatrix}
            \bv_0 & \bv_1 & \dots & \bv_d\\
            1   &  1  & \dots & 1
        \end{bmatrix}\right|
        \sum_{j=0}^d f(\bv_j) - \int_{J} f(\bx)d\bx   ~.
    \end{align*}
    The second equation follows from Lemma \ref{lem:int_mu}. Therefore,
    $$
    \vol(P(f,J)) = \frac{1}{(d+2)!}\left|\det\begin{bmatrix}
        \bv_0 & \bv_1 & \dots & \bv_d\\
        1   &  1  & \dots & 1
    \end{bmatrix}\right|
    \sum_{j=0}^d f(\bv_j) - \frac{1}{d+2}\int_{J} f(\bx)d\bx ~.
    $$
    \qed
\end{proof}

Theorem \ref{thm:perspecvol} reduces calculation of $\vol(\hyperlink{Pfj}{P(f,J)} )$
to the calculation of the integral $\int_{J} f(\bx)d\bx$. We will 
make a detailed exploration of the
fundamental problem of integration over a simplex in Section \ref{sec:int_simplex}.
We note that calculation of the volume of \hyperlink{P0fj}{$P^0(f,J)$} is generally more complicated than the integral $\int_{J} f(\bx)d\bx$ (see Theorem \ref{thm:naivevol}).
We address some relevant special cases in Section \ref{sec:multi}.
\begin{theorem}\label{thm:naivevol}
    Suppose that $f(\mathbf{0})=0$ and $f$ is continuous and convex on $\conv(J\cup \{\mathbf{0}\})$, where the $d$-simplex $J:=\conv\{\bv_0,\allowbreak \bv_1, \dots,\bv_{d}\} \subset \R^d_{\geq 0}\setminus \{ \mathbf{0}\}$. Then
 \begin{align*}
   \vol(\hyperlink{P0fj}{P^0(f,J)} )) 
    &=\frac{1}{(d+2)!}\left|\det\begin{bmatrix}
        \bv_0 & \bv_1 & \dots & \bv_d\\
        1   &  1  & \dots & 1
    \end{bmatrix}\right|
    \sum_{j=0}^d f(\bv_j) - \int_0^1 z^d\int_{J} f(z\bx) d\bx dz. 
    \end{align*}  
\end{theorem}
\begin{proof}
   Considering the definition of the na\"{i}ve relaxation, its volume is
    \begin{align*}
   \vol(\hyperlink{P0fj}{P^0(f,J)} ))  &= \iint_{\bx\in z\cdot J,0\le z\le 1} (\mu(\bx,z) - f(\bx))d\bx dz \\
    &= \int_{0\le z\le 1}z^d\int_{\tilde{\bx}\in J}(z\mu(\tilde{\bx}) - f(z\tilde{\bx})) d\tilde{\bx} dz\\
    &= \int_{0\le z\le 1}z^{d+1}dz\int_{\tilde{\bx}\in J}\mu(\tilde{\bx}) d\tilde{\bx} - \int_0^1 z^d\int_{J} f(z\bx) d\bx dz.\\
    &=\frac{1}{(d+2)!}\left|\det\begin{bmatrix}
        \bv_0 & \bv_1 & \dots & \bv_d\\
        1   &  1  & \dots & 1
    \end{bmatrix}\right|
    \sum_{j=0}^d f(\bv_j) - \int_0^1 z^d\int_{J} f(z\bx) d\bx dz.
    \end{align*}
    The last equation follows from Lemma \ref{lem:int_mu}.
\qed \end{proof}
\begin{remark}
Theorem \ref{thm:naivevol} provides a different formula to compute the volume of the na\"{i}ve relaxation from \cite[Thm. 2 and Cor. 4]{lee_gaining_2020}, by slicing along variable $z$ instead of variable $y$.
\end{remark}

\section{Integration over a simplex}\label{sec:int_simplex}

Integration over a simplex is a well-researched topic.
%
%
In this section, we survey the main results for integrating over a simplex. We add some missing details to the proofs in the literature and unify the notations to make the results more accessible to the optimization community. Our goal is to provide the most general results to serve as a toolbox for future research.

\cite[Thm. 1 and Cor. 3]{baldoni_how_2010} proved that integrating polynomials over a simplex is NP-hard, while there exists a polynomial-time algorithms for integrating polynomials of fixed total degree.
We are mainly interested in closed-form formulae to be carried to later analyses for the integration of
some convex function over a general simplex, e.g., power of linear forms $(\bc^\top\bx)^q$ and exponential of linear functions $e^{\bc^\top\bx}$.
This section is organized by method, including the monomial formula, series expansion, symmetric multilinear form, Fourier transformation, and cubature rules. Table \ref{tab:summary_integrate_simplex} summarizes the functions considered as the integrands in this section, where  \emph{generalized polynomials} refers to  sums of monomials whose exponents are nonnegative real numbers, and  \emph{symmetric multilinear forms} refers to functions $H:(\mathbb{R}^d)^q\rightarrow \R$ such that $H(\bx_1,\bx_2,\dots,\bx_q) = H(\bx_{j_1},\bx_{j_2},\dots,\bx_{j_q})$ for any permutation $(j_1,j_2,\dots,j_q)$ of $(1,2,\dots,q)$ (symmetric) and $H(\lambda\bx_1+\lambda_y\by, \bx_2,\dots,\bx_q) =  \lambda H(\bx_1, \bx_2,\dots,\bx_q)+\lambda_y H(\by, \bx_2,\dots,\bx_q)$ for any $\lambda,\lambda_y\in\mathbb{R}$ (multilinear).

\begin{table}[H]
    \centering
    \begin{tabular}{c|c|c}
    Integrand & Simplex & Subsection  \\
    \hline\hline
    & & \phantom{$\cdot$}\\[-6pt]
    Generalized polynomials & \multirow{2}{*}{$\Delta_d$} & \multirow{2}{*}{\ref{subsec:mono}}\\
    $\sum_{t\in\mathcal{T}}x_1^{\alpha^{t}_1}x_2^{\alpha^{t}_2}\dots x_d^{\alpha^{j}_d}$, where $\alpha_j^t\in\mathbb{R}_{\ge0}$ &  \\[3pt]
    \hline
        & & \phantom{$\cdot$}\\[-6pt]
    Polynomials & $J$ &  \ref{subsec:mono}, \ref{subsec:expansion}, \ref{subsec:sml}, \ref{subsec:cubature}\\[3pt]
    \hline
        & & \phantom{$\cdot$}\\[-6pt]
    Power of affine forms $(\bc^\top\bx+b)^n$, $n\in\mathbb{Z}_{\ge1}$ & $J$ &  \ref{subsec:expansion}\\[3pt]
    \hline
        & & \phantom{$\cdot$}\\[-6pt]
    Symmetric multilinear form & \multirow{2}{*}{$J$} & \multirow{2}{*}{ \ref{subsec:sml}}\\
    $H(\bx_1,\bx_2,\dots,\bx_q)$&  \\[3pt]
    \hline
        & & \phantom{$\cdot$}\\[-6pt]
    Exponentials of affine functions $e^{\bc^\top\bx+b}$ & $J$ &  \ref{subsec:expansion}, \ref{subsec:fourier}
    \end{tabular}
    \caption{Summary for the functions considered in each subsection}
    \label{tab:summary_integrate_simplex}
\end{table}

\subsection{Monomial formula over a standard simplex}\label{subsec:mono}
There is a well-known formula to integrate a particular generalized polynomial over the standard $d$-simplex $\Delta_d$ in $\R^d$ (see, for example \cite{lasserre_simple_2020}).
\begin{proposition}\label{prop:mono_standard_simplex}
\begin{equation*}
\int_{\Delta_d} x_1^{\alpha_1}x_2^{\alpha_2}\dots x_d^{\alpha_d} (1-x_1-\dots - x_d)^{\alpha_{d+1}} d\bx = \frac{\prod_{j=1}^{d+1}\Gamma(\alpha_j+1)}{\Gamma\left(\sum_{j=1}^{d+1}{\alpha_j}+d+1\right)},
\end{equation*}
where $\alpha_j\in \R$, $\alpha_j>-1$, and the usual gamma function $\Gamma(z):=\int_0^\infty x^{z-1} \mathrm{e}^{-x}dx$ for $z>0$.
\end{proposition}
%
%

To integrate a polynomial over a standard simplex, we can represent it as a sum of monomials and then employ Proposition \ref{prop:mono_standard_simplex} for each monomial, noting that we can take $\alpha_{d+1}=0$ (see for example \cite[Eqn. 2.3]{grundmann1978invariant}).
This idea also appears in \cite{lasserre_simple_2020} with a different interpretation associated with ``Bombieri-type polynomials''.

\subsection{Series expansion}\label{subsec:expansion}
\cite[Sec. 3.2]{baldoni_how_2010} provides several polynomial-time algorithms to calculate the exact integral of a polynomial with fixed degree 
over a general simplex. 
These algorithms are based on  integration formulae over a general simplex, for 
powers of linear functions  with positive integer exponent
and for the exponential of linear functions.
\begin{lemma}[{\cite[Lem. 8 and Rem. 9]{baldoni_how_2010}}]\label{lem:simplex} Suppose that $J:=\conv\{\bv_0,\allowbreak\bv_1,\dots,\bv_d\}\subset \R^d$. Then we have
\begin{align}
    \int_{J}(\bc^\top \bx)^n d\bx &= d! \vol(J)\frac{n!}{(n+d)!} \sum_{\bk\in\Z_{\ge0}^{d+1},\atop ~\|\bk\|_1=n} (\bc^\top \bv_0)^{k_0}\dots(\bc^\top \bv_d)^{k_d}~,\notag\\[5pt]
    \int_{J} \mathrm{e}^{\bc^\top \bx} d\bx &= d! \vol(J) \sum_{\bk\in\Z_{\ge0}^{d+1}} \frac{(\bc^\top \bv_0)^{k_0}\dots(\bc^\top \bv_d)^{k_d}}{(\|\bk\|_1+d)!}. \notag
\end{align}
\end{lemma}
By treating the sequence $\frac{(n+d)!}{n!}\int_{J}(\bc^\top \bx)^n d\bx$ as the coefficient of a formal power series, \cite{baldoni_how_2010} gives a useful generating function using Lemma \ref{lem:simplex}:
%
\begin{theorem}[{\cite[Thm. 10]{baldoni_how_2010}}]\label{thm:taylor_series}
Suppose that $J:=\conv\{\bv_0,\bv_1,\allowbreak\dots,\bv_d\}\subset \R^d$. Then we have
\begin{equation*}\label{eqn:taylor_series}
d!\vol(J)\frac{1}{\prod_{j=0}^{d}(1-t(\bc^\top \bv_j))}
=
\sum_{n=0}^{\infty} \left[\frac{(n+d)!}{n!}\int_{J} (\bc^\top \bx)^n d\bx \right]t^n .
\end{equation*}
\end{theorem}
Therefore, we can obtain the ``short formulae'' of Brion (in the case of a simplex) under a genericity assumption.
\begin{theorem}[\cite{brion1988points}]
    Suppose that $J:=\conv\{\bv_0,\bv_1,\dots,\bv_d\}\subset \R^d$, and $\bc^\top \bv_j \ne \bc^\top \bv_k$ for all $j\ne k$. Then we have
    \begin{align}
        \int_{J}(\bc^\top \bx)^n d\bx &= d! \vol(J)\frac{n!}{(n+d)!} \sum_{j=0}^d \frac{(\bc^\top \bv_j)^{n+d}}{\prod\limits_{k \,:\, k\neq j} \bc^\top(\bv_j-\bv_k)}, \label{eqn:linear_short}\\
        \int_{J} \mathrm{e}^{\bc^\top \bx} d\bx &= d! \vol(J) \sum_{j=0}^d \frac{\mathrm{e}^{\bc^\top \bv_j}}{\prod\limits_{k \,:\, k\neq j} \bc^\top(\bv_j-\bv_k)}.\label{eqn:exp_short}
    \end{align}
\end{theorem}
\begin{proof}
We can decompose the following function into partial fractions and apply the expansion of geometric series:
\begin{align*}
    \frac{1}{\prod_{j=0}^{d}(1-t(\bc^\top \bv_j))} &=\sum_{j=0}^d\frac{1}{1-t(\bc^\top \bv_j)}\frac{(\bc^\top \bv_j)^d}{\prod\limits_{k \,:\,  k\neq j} \bc^\top (\bv_j-\bv_k)}\\
    &= \sum_{j=0}^d\left(\sum_{n=0}^{\infty}t^n(\bc^\top \bv_j)^n \right)\frac{(\bc^\top \bv_j)^d}{\prod\limits_{k \,:\,  k\neq j} \bc^\top (\bv_j-\bv_k)}\\
    &= \sum_{n=0}^{\infty}\left[\sum_{j=0}^d\frac{(\bc^\top \bv_j)^{n+d}}{\prod\limits_{k \,:\,  k\neq j} \bc^\top (\bv_j-\bv_k)}\right]t^n.
\end{align*}
Then \eqref{eqn:linear_short} immediately follows from Theorem \ref{thm:taylor_series}. For \eqref{eqn:exp_short}, we have
\begin{align*}
 \int_{J} \mathrm{e}^{\bc^\top \bx} d\bx &= \sum_{n=0}^{+\infty} \int_{J} \frac{(\bc^\top \bx)^n}{n!} d\bx \\
 &= d! \vol(J)\sum_{n=0}^{+\infty}\frac{1}{(n+d)!} \sum_{j=0}^d \frac{(\bc^\top \bv_j)^{n+d}}{\prod\limits_{k \,:\,  k\neq j} \bc^\top(\bv_j-\bv_k)}\\
 &= d! \vol(J) \sum_{j=0}^d \frac{1}{\prod\limits_{k \,:\,  k\neq j} \bc^\top(\bv_j-\bv_k)}\sum_{n=0}^{+\infty}\frac{(\bc^\top \bv_j)^{n+d}}{(n+d)!}\\
 &= d! \vol(J) \sum_{j=0}^d \frac{\mathrm{e}^{\bc^\top \bv_j}}{\prod\limits_{k \,:\,  k\neq j} \bc^\top(\bv_j-\bv_k)}.
\end{align*}
where the last equation follows from the fact that
\begin{align}\label{fact}
\sum_{j=0}^d \frac{(\bc^\top \bv_j)^n}{\prod\limits_{k \,:\,  k\neq j} \bc^\top(\bv_j-\bv_k)} = 0, ~\text{for}~ n = 0,1,\ldots, d-1.
\end{align}
This is due to the fact that the remainder of Lagrange interpolation polynomials is zero for polynomials of degree at most $d$:
\begin{align}\label{interpolation}
t^n - \sum_{j=0}^d \frac{\prod\limits_{k \,:\,  k\neq j}(t- a_k)}{\prod\limits_{k \,:\,  k\neq j}(a_j-a_k)}a_j^n = 0,~\text{for}~ n = 1,\ldots, d.
\end{align}
Letting $a_j:=1/(\bc^\top \bv_j)$ in \eqref{interpolation} and evaluating at $t=0$, we get \eqref{fact}:
\begin{align*}
    t^n &= \sum_{j=0}^d \frac{\prod\limits_{k \,:\,  k\neq j}(\frac{1}{\bc^\top \bv_k}-t)}{\prod\limits_{k \,:\,  k\neq j}(\frac{1}{\bc^\top \bv_k}-\frac{1}{\bc^\top \bv_j})}\left(\frac{1}{\bc^\top \bv_j}\right)^n\\
     &= \sum_{j=0}^d \frac{\prod\limits_{k \,:\,  k\neq j}(1- (\bc^\top \bv_k)t)}{\prod\limits_{k \,:\,  k\neq j}(\bc^\top \bv_j-\bc^\top \bv_k)}(\bc^\top \bv_j)^{d-n}\\
\Rightarrow~    0 & = \sum_{j=0}^d \frac{(\bc^\top \bv_j)^{d-n}}{\prod\limits_{k \,:\,  k\neq j}(\bc^\top \bv_j-\bc^\top \bv_k)}, ~\text{for}~ n = 1,\ldots, d.
\end{align*}
\qed \end{proof}
In the general case, when the genericity assumption ($\bc^\top \bv_j \ne \bc^\top \bv_k$ for all $j\ne k$) fails, we can take $K\subset \{0,1,\ldots,d\}$  to be an index set of different poles $t=1/(\bc^\top \bv_k)$, and for $k\in K$, we let $m_k:=|\{j\in\{0,1,\ldots,d\}:~\bc^\top \bv_j = \bc^\top \bv_k\}|$, which is the order of the pole.
\begin{cor}[{\cite[Cor. 13]{baldoni_how_2010}}]\label{cor:residue}Suppose that\\
$J:=\conv\{\bv_0,\bv_1,\dots,\bv_d\}\subset \R^d$. Then we have
    \begin{align*}
    &\int_{J}(\bc^\top \bx)^n d\bx \\
    &\quad = d! \vol(J)\frac{n!}{(n+d)!} \sum_{k\in K}\mathrm{Res}\left(\frac{(\epsilon+\bc^\top \bv_k)^{n+d}}{\epsilon^{m_k}\prod\limits_{j\in K\setminus\{k\}} (\epsilon+\bc^\top(\bv_k-\bv_j))^{m_j}},~\epsilon=0\right),
    \end{align*}
where $\mathrm{Res}$ denotes the residue (here at $\epsilon=0$).
\end{cor}
\noindent If there are poles with high order, the residue can be calculated by the Laurent series expansion. 


Next, we establish an affine generalization of Theorem \ref{thm:taylor_series}. 
\begin{theorem}\label{eqn:taylor_series_2}
Suppose that $J:=\conv\{\bv_0,\bv_1,\dots,\bv_d\}\subset \R^d$, we have
\begin{equation*}
d!\vol(J)\frac{1}{\prod_{j=0}^{d}(1-t(\bc^\top \bv_j+b))}
=
\sum_{n=0}^{\infty} \left[\frac{(n+d)!}{n!}\int_{J} (\bc^\top \bx+b)^n d\bx \right]t^n .
\end{equation*}
\end{theorem}
\begin{proof}
We first prove an affine generalization of Lemma \ref{lem:simplex},
\begin{align}\label{eqn:simplex_affine}
\frac{(n+d)!}{n!}\int_{J}(\bc^\top \bx+b)^n d\bx &= d! \vol(J) \sum_{\bk\in\Z_{\ge0}^{d+1},~ \|\bk\|_1=n} \prod_{j=0}^d (\bc^\top \bv_j+b)^{k_j}.
\end{align}
By the binomial theorem and Lemma \ref{lem:simplex}, we have
\begin{align*}
&\frac{(n+d)!}{n!}\int_{J}(\bc^\top \bx+b)^n d\bx = \frac{(n+d)!}{n!}\sum_{m=0}^n\binom{n}{m} b^{n-m}\int_{J}(\bc^\top\bx)^m d\bx\\
=~& d!\vol(J)\sum_{m=0}^n\frac{(n+d)!}{(n-m)!(m+d)!} b^{n-m}\sum_{\bk\in\Z_{\ge0}^{d+1},~\|\bk\|_1=m} \prod_{j=0}^d(\bc^\top \bv_j)^{k_j}.
\end{align*}
%
%
Now we apply the binomial theorem to the right-hand-side of \eqref{eqn:simplex_affine} and collect terms with the same degree of $b$:
\begin{align*}
~& d! \vol(J) \sum_{\bk\in\Z_{\ge0}^{d+1},\atop \|\bk\|_1=n} \sum_{m=0}^{n} b^{n-m}  \sum_{\mathbf{\alpha}\in \Z_{\ge 0 }^{d+1},~\atop \|\mathbf{\alpha}\|_1=n-m} \prod_{j=0}^d \binom{k_j}{\alpha_j}(\bc^\top \bv_j)^{k_j-\alpha_j}\\
=~& d! \vol(J) \sum_{m=0}^n b^{n-m} \sum_{\bk\in\Z_{\ge0}^{d+1},\atop \|\bk\|_1=m} \left[\sum_{\mathbf{\alpha}\in \Z_{\ge 0 }^{d+1},~\atop \|\mathbf{\alpha}\|_1=n-m} \prod_{j=0}^d \binom{k_j+\alpha_j}{\alpha_j}\right]\prod_{j=0}^{d}(\bc^\top \bv_j)^{k_j}.
\end{align*}
It can be shown to be the same as the right-hand side by the well-known generalized Vandermonde identity: 
%
$$
\sum_{\mathbf{\alpha}\in \Z_{\ge 0 }^{d+1},~\|\mathbf{\alpha}\|_1=n-m} \prod_{j=0}^d \binom{k_j+\alpha_j}{\alpha_j} = \binom{\|\bk\|_1+n-m+d}{n-m}.
$$

A short proof of this identity
is by double counting the coefficient of $t^{n-m}$ on both sides of $\prod_{j=0}^d (1+t)^{-(k_j+1)} = (1+t)^{-(\|\bk\|_1+d+1)}$.

Therefore, we can verify that \eqref{eqn:simplex_affine} holds.
%

%
Then by multiplying \eqref{eqn:simplex_affine} by $t^n$ and summing from 0 to $\infty$, we have
\begin{align*}
&\sum_{n=0}^{\infty} \left[\frac{(n+d)!}{n!}\int_{J} (\bc^\top \bx+b)^n d\bx \right]t^n \\
= & d! \vol(J)\sum_{n=0}^{\infty} t^n\sum_{\bk\in\Z_{\ge0}^{d+1},\atop \|\bk\|_1=n} (\bc^\top \bv_0+b)^{k_0}\dots(\bc^\top \bv_d+b)^{k_d}\\
= & d! \vol(J)\prod_{j=0}^d\sum_{k_j=0}^{\infty} (\bc^\top \bv_j + b)^{k_j} t^{k_j}\\
= & d!\vol(J)\frac{1}{\prod_{j=0}^{d}(1-t(\bc^\top \bv_j+b))}.
\end{align*}
Therefore, the result holds.
\qed \end{proof}
 Using the same method as \eqref{eqn:linear_short}, we obtain the following corollary.
\begin{cor}
Suppose that $J:=\conv\{\bv_0,\bv_1,\dots,\bv_d\}\subset \R^d$. Suppose further that $\bc^\top \bv_j + b\ne \bc^\top \bv_k + b$.
Then we have
\begin{align*}
    \int_{J}(\bc^\top \bx+b)^n d\bx &= d! \vol(J)\frac{n!}{(n+d)!} \sum_{j=0}^d \frac{(\bc^\top \bv_j+b)^{n+d}}{\prod\limits_{k \,:\, k\neq j} \bc^\top(\bv_j-\bv_k)}.
\end{align*}
\end{cor}

We can extend Theorem \ref{eqn:taylor_series_2} to calculate the integration of a product of powers of $D$ affine forms.
By replacing $t\bc$ by $\sum_{j=1}^D t_j \bc_j$ and $tb$ by $\sum_{j=1}^D t_j b_j$ in
Theorem \ref{eqn:taylor_series_2}
and taking the expansion in powers $t_1^{\alpha_1}\dots t_D^{\alpha_D}$, we obtain:

\begin{cor}\label{cor:duality_D}
Suppose that $J:=\conv\{\bv_0,\bv_1,\dots,\bv_d\}\subset \R^d$, we have
\begin{align*}
&    d!\vol(J)\frac{1}{\prod_{j=0}^{d}(1-t_1(\bc_1^\top \bv_j+b_1)-\dots-t_D(\bc_D^\top \bv_j+b_D))}
\notag\\
&\quad =\sum_{\mathbf{\alpha}\in\Z_{\ge0}^D} \left[\frac{(\|\mathbf{\alpha}\|_1+d)!}{\alpha_1!\dots\alpha_D!}\int_{J} (\bc_1^\top \bx+b_1)^{\alpha_1}\dots (\bc_D^\top \bx+b_D)^{\alpha_D} d\bx \right]t_1^{\alpha_1}\dots t_D^{\alpha_D} .
\end{align*}
\end{cor}

Next, we review the three algorithms implemented in \cite{baldoni_how_2010}.

The first algorithm is called the \emph{Taylor-expansion method}.
By Corollary \ref{cor:duality_D}, letting $D:=d$ and $\bc_j := \be_j$, $b_j:=0$, for $j=1,\dots,d$, we can compute the integral 
$\int_{J}x_1^{\alpha_1}\dots x_d^{\alpha_d}d\bx$
of a monomial  by multiplying $\frac{d!\vol(J)\alpha_1!\dots \alpha_d!}{(\|\mathbf{\alpha}\|_1+d)!}$ by the coefficient of $t_1^{\alpha_1}\dots t_d^{\alpha_d}$ in the Taylor expansion of 
$$
\frac{1}{\prod_{j=0}^{d}(1-\bt^\top \bv_j)}.
$$
Using this method, we can integrate a polynomial with fixed degree $q$ in polynomial time (see \cite[Proof of Cor. 3]{baldoni_how_2010}).

The second algorithm is called the  \emph{linear-form decomposition method}. We first decompose arbitrary monomial as sums of powers of linear forms 
\begin{align*}
    &x_1^{\alpha_1}x_2^{\alpha_2}\dots x_d^{\alpha_d} = \\
    &\qquad \frac{1}{(\|\mathbf{\alpha}\|_1)!}\sum_{0\le \beta_i\le \alpha_i} (-1)^{\|\mathbf{\alpha}\|_1-\|\mathbf{\beta}\|_1}\binom{\alpha_1}{\beta_1}\dots\binom{\alpha_d}{\beta_d}(\beta_1x_1+\dots+\beta_dx_d)^{\|\mathbf{\alpha}\|_1},\nonumber 
\end{align*}
and then we use Corollary \ref{cor:residue} to compute the integral of each linear form. This is also a polynomial-time algorithm for fixed degree $q$ (See \cite[Alternative proof of Cor. 3]{baldoni_how_2010}).

The third algorithm is called the \emph{iterated-Laurent method}. The equation \eqref{eqn:exp_short} can also be viewed as an equation with respect to the variables $\mathbf{c}$.
We know that $\int_{J} x_1^{\alpha_1}\dots x_d^{\alpha_d}d\bx$ is the coefficient of $\frac{c_1^{\alpha_1}\dots c_d^{\alpha_d}}{\alpha_1!\dots\alpha_d!}$ in the Taylor expansion of the left-hand side of \eqref{eqn:exp_short} $\int_{J} \mathrm{e}^{\bc^\top \bx}d\bx$. 
By expanding the right-hand side of \eqref{eqn:exp_short} into an iterated Laurent series with respect to the variables $c_1,\dots,c_d$, we can compute the integral by comparing the coefficient of $\frac{c_1^{\alpha_1}\dots c_d^{\alpha_d}}{\alpha_1!\dots\alpha_d!}$ (see \cite[Rem. 15]{baldoni_how_2010}).

From the numerical experiment in \cite{baldoni_how_2010}, we know that for low dimensions $(d\le 5)$, the iterated-Laurent method is faster than the two other methods; for high dimensions, the linear-form decomposition method is faster than the other two methods.

\subsection{Symmetric multilinear form}\label{subsec:sml}
A (multivariate) polynomial $f(\bx)$ is \emph{$q$-homogeneous} if $f(\lambda \bx) = \lambda^q f(\bx)$.
Besides representing a polynomial as a sum of monomials, we may also write a polynomial as a sum of homogeneous polynomials.
\cite{lasserre_multi-dimensional_2001} provides a nice formula for the integration of a $q$-homogeneous ($q$ is a positive integer)
polynomial on a simplex by associating with the symmetric multilinear form.
\begin{lemma}\label{lem:polarization}
For a $q$-homogeneous polynomial $f(x): \R^d \rightarrow \R$, there exists a symmetric multilinear form $H_f: (\R^d)^q \rightarrow \R$ by the polarization formula
\begin{equation*}
H_f(\bx_1,\bx_2,\dots,\bx_q) = \frac{1}{2^q q!}\sum_{\mathbf{\epsilon}\in\{\pm1\}^q} \epsilon_1\epsilon_2\dots\epsilon_q f(\sum_{j=1}^q\epsilon_j \bx_j),
\end{equation*}
such that $H_f(\bx,\bx,\dots,\bx) = f(\bx)$. 
\end{lemma}
\begin{proof}
The symmetry follows from the definition of $H_f$ because any permutation between the $x_j$'s would result in the same $H_f$. 
Given $q$-homogeneous $f(x)$, we can
easily check that $H_f(\bx,\bx,\dots,\bx) = f(\bx)$: 
\begin{align*}
    H_f(\bx,\bx,\dots,\bx) & = \frac{1}{2^q q!}\sum_{\mathbf{\epsilon}\in\{\pm1\}^q} \epsilon_1\epsilon_2\dots\epsilon_q f(\sum_{j=1}^q\epsilon_j \bx)\\
    & = f(\bx)\frac{1}{2^q q!}\sum_{\mathbf{\epsilon}\in\{\pm1\}^q} \epsilon_1\epsilon_2\dots\epsilon_q (\sum_{j=1}^q\epsilon_j)^q\\
    & = f(\bx)\frac{1}{2^q q!}\sum_{k=0}^{q} (-1)^k \binom{q}{k}(q-2k)^q\\
    & = f(\bx)\frac{1}{2^q q!}\sum_{j=0}^q\binom{q}{j}q^{q-j} (-2)^j\sum_{k=0}^{q} (-1)^k \binom{q}{k}k^j\\
    & = f(\bx).
\end{align*}
The last equation follows from a well-known identity (see \cite{ruiz2004algebraic}, for example):
\vbox{
$$
\sum_{k=0}^{q} (-1)^k \binom{q}{k} k^j =\left\{
\begin{array}{ll}
    0, &\quad 0\le j\le q-1; \\
    q!(-1)^q, &\quad  j=q.
\end{array}
\right.
$$
\qed
}\end{proof}

\begin{theorem}[{\cite[Thm. 2.1]{lasserre_multi-dimensional_2001}}]\label{thm:qhomo}
    Suppose that $J:=\conv\{\bv_0,\bv_1,\dots,\bv_d\}\allowbreak\subset \R^d$. Suppose that $H:(\R^d)^q \rightarrow \R$ is a symmetric multilinear form.
    Then we have
    \begin{equation}\label{eqn:symlinear}
    \int_{J} H(\bx,\bx,\dots,\bx) d\bx = \frac{\vol(J)}{\binom{q+d}{q}}\sum_{0\le i_1\le\dots\le i_q\le d} H(\bv_{i_1},\bv_{i_2},\dots,\bv_{i_q}).
    \end{equation}
    Suppose that $f:\R^d \rightarrow \R$ is a $q$-homogeneous polynomial ($q\in\mathbb{Z}_{\ge 1}$), then 
    \begin{equation}\label{eqn:qhomo}
    \int_{J} f(\bx)d\bx = \frac{\vol(J)}{2^q q!\binom{q+d}{q}}\sum_{0\le i_1\le\dots\le i_q\le d~} \sum_{\mathbf{\epsilon}\in\{\pm1\}^q} \epsilon_1\dots\epsilon_q f\left(\sum_{j=1}^q\epsilon_j \bv_{i_j}\right).
    \end{equation}
\end{theorem}
\begin{proof}
For \eqref{eqn:symlinear}, we make an affine bijection between $J$ and the standard 
$d$-simplex $\Delta_d$ via \eqref{eqn:affine_transformation_B}, 
$$
\bx\in J \quad \Leftrightarrow \quad \bt= \bB^{-1}(\bx-\bv_0) \in \Delta_d.
$$
Then the integration becomes
\begin{align*}
    &\int_{J} H(\bx,\dots,\bx) d\bx =|\det(\bB)|\int_{\Delta_d} H(\bB\bt + \bv_0,\dots,\bB\bt + \bv_0) d\bt \\
    =&~ d!\vol(J) \int_{\Delta_d} H\left(\sum_{j=1}^d t_j \bv_j  + (1-\sum_{j=1}^d t_j) \bv_0,\dots, \sum_{j=1}^d t_j \bv_j  + (1-\sum_{j=1}^d t_j) \bv_0\right) d\bt\\
    =&~ d!\vol(J) \!\!\!\sum_{\substack{\mathbf{\alpha}\in \Z_{\ge0}^{d+1}\\\|\mathbf{\alpha}\|_1= q}} \frac{q!}{\alpha_0!\dots\alpha_d!}H(\bv_0^{\alpha_0},\dots,\bv_d^{\alpha_d})\!\! \int_{\Delta_d}\!\! t_1^{\alpha_1}\dots t_d^{\alpha_d} (1-t_1-\dots - t_d)^{\alpha_0} d\bt
\end{align*}
where the last equality follows from the multilinearity and symmetry of $H$, and 
$$
H(\bv_0^{\alpha_0},\dots,\bv_d^{\alpha_d}):=H(\underbrace{\bv_0,\dots,\bv_0}_{\alpha_0~\text{times}},\dots, \underbrace{\bv_d,\dots,\bv_d}_{\alpha_d~\text{times}}).
$$ 
Therefore, by Proposition \ref{prop:mono_standard_simplex}, we get \eqref{eqn:symlinear}.
For \eqref{eqn:qhomo}, we can use Lemma \ref{lem:polarization} to construct the associated $H_f$. Then it follows from \eqref{eqn:symlinear} directly.
\qed \end{proof}

By Theorem \ref{thm:qhomo}, we can integrate a polynomial with fixed degree $q$ in polynomial time.\\

\subsection{Fourier transformation}\label{subsec:fourier}
 There is a  Fourier-transformation method in \cite{barvinok1991computation,barvinok1994computation} for integration of a class of exponential functions over a polytope in  standard form. 
 We can recover \eqref{eqn:exp_short} via this method.
\begin{theorem}[\cite{barvinok1991computation,barvinok1994computation}]\label{thm:fourier}
Let $P:=\{\bx\in\R^n_{\geq 0} ~:~ \bA\bx=\bb\}$, where 
 $\bA\in\R^{m\times n}$ has $\rank \bA=m<n$, $\bb\in\R^m$, and  
 suppose that $\dim P = n-m$.
That is, $P$ is full dimensional in the $(n-m)$-dimensional hyperplane $\{\bx\in\R^n ~:~ \bA\bx=\bb\}$. 
Then, for all  $\bc\in \R^n_{>0} $~, we have
$$
\int_{P} \mathrm{e}^{-\bc^\top \bx} d\bx = \sqrt{\det(\bA\bA^\top)} \int_{\R^m} \mathrm{e}^{2\pi i \bb^\top \by} \prod_{j=1}^n \frac{1}{2\pi i \ba_j^\top \by + \bc_j} d\by,
$$
where $\ba_1,\dots,\ba_n$ are the columns of $\bA$.
\end{theorem}

\begin{proof}
For $\varphi:\R^n \rightarrow \R$,
we denote the Fourier transform by $\hat{\varphi}(\xi) :=\int \mathrm{e}^{-2\pi i \xi^\top \bx} \varphi(\bx) d\bx$. 
For
$$
g(\bx) := \left\{
\begin{array}{ll}
\mathrm{e}^{-\bc^\top \bx}, &\quad \bx\in \R^n_{\ge0}~;\\
0,&\quad  \text{otherwise,}
\end{array}
\right. 
$$
we have
$$
\hat{g}(\by) = \int_{\R^n} \mathrm{e}^{-2\pi i \by^\top \bx} g(\bx)d\bx = \int_{\R^n_{\ge0}} \mathrm{e}^{-2\pi i\by^\top \bx - \bc^\top \bx}d\bx = \prod_{j=1}^n  \frac{1}{2\pi i y_j + c_j}.
$$
Choose $\bx_0\in\R^n$  such that  $\bA\bx_0 =\bb$. 
Let $f(\bx) := g(\bx+\bx_0)$. Then
$$
\hat{f}(\by) = \mathrm{e}^{2\pi i \bx_0^\top \by}\hat{g}(\by)=\mathrm{e}^{2\pi i \bx_0^\top \by}\prod_{j=1}^n  \frac{1}{2\pi i y_j + c_j}.
$$
We can see that  
$$
\mathrm{e}^{2\pi i \bx_0^\top \by}\hat{g}(\by) =\int_{L+\bx_0}g(\bx) d\bx =\int_{L} g(\bx+\bx_0) d\bx = \int_{L} f(\bx) d\bx,
$$
where $L:=\{\bx\in \R^n:~\bA\bx=\mathbf{0}\}$ denotes the null space of $\bA$.

By the formula 
$$
\int_{L} f(\bx) d\bx =\int_{L^\perp} \hat{f}(\by) d\by,
$$
where  $L^{\perp}:=\{\bA^\top \bx:\bx\in\R^m\}$ denotes the row space of $\bA$
(see \cite{becnel2007infinite,hormander_analysis_2003}), 
%
%
we obtain
\begin{align*}
\int_P \mathrm{e}^{-\bc^\top \bx} d\bx & = \int_{L^{\perp}} \hat{f}(\by) d\by\\
&=\int_{L^{\perp}} \mathrm{e}^{2\pi i \bx_0^\top\by} \prod_{j=1}^n  \frac{1}{2\pi i y_j + c_j} d\by\\
&= \sqrt{\det(\bA\bA^\top)}\int_{\R^m} \mathrm{e}^{2\pi i x_0^\top A^\top \bz} \prod_{j=1}^n  \frac{1}{2\pi i \ba_j^\top \bz + c_j} d\bz\\
&= \sqrt{\det(\bA\bA^\top)}\int_{\R^m} \mathrm{e}^{2\pi i b^\top\bz} \prod_{j=1}^n  \frac{1}{2\pi i \ba_j^\top \bz + c_j} d\bz~,
\end{align*}
where the penultimate equation follows from a change of variables $\by:=\bA^\top \bz$. 
\qed \end{proof}

We can connect the integration over $\Delta_d$ in $\R^d$ with the integration over $\Delta'_d := \{\bx=(x_0,x_1,x_2,\dots,x_d)\in\R^{d+1}_{\ge 0}: ~\sum_{j=0}^d x_j=1\}$ in $\R^{d+1}$, via
\begin{equation*}
\sqrt{d+1}\int_{\Delta_d}  f(1-\sum_{j=1}^d x_j,x_1,x_2,\dots,x_d) d\bx =\int_{\Delta'_d} f(x_0,x_1,x_2,\dots,x_d) d\bx~.
\end{equation*}
%
To obtain the above equation, we observe that the affine transformation $\varphi: \Delta_d \rightarrow \Delta'_d$ given by 
\begin{equation}\label{eqn:affine_transformation_Q}
\varphi(\bx) :=
\begin{bmatrix}
-\mathbf{1}_d^\top\\
I_d
\end{bmatrix}\bx + \begin{bmatrix}
1\\
\mathbf{0}_d
\end{bmatrix}
=:\bQ\bx+\gamma,
\end{equation}
satisfies $\bx\in \Delta_d \Leftrightarrow \varphi(\bx) \in \Delta'_d$.
Therefore, by performing the affine transformation $\varphi$, the ratio of the volumes is
$\sqrt{\det(\bQ^\top \bQ)}$ (see \cite{gover2010determinants}, for example), which here becomes
$$
\sqrt{\det(\bQ^\top \bQ)} = \sqrt{\det(\mathbf{1}_d\mathbf{1}_d^\top+I_d)} = \sqrt{d+1}.
$$

\begin{cor}[\cite{barvinok1991computation,barvinok1994computation}]\label{cor:recover}
We can recover \eqref{eqn:exp_short} from Theorem \ref{thm:fourier}.
\end{cor}

\begin{proof}
We first show that for $\Delta_d'=\{\bx\in\R^{d+1}_{\ge 0}: ~\sum_{j=0}^d x_j=1\}$ in $\R^{d+1}$, $c_j\ne c_k$ for all $j\ne k$, 
$$
\int_{\Delta_d'} \mathrm{e}^{\bc^\top \bx} d\bx = \sqrt{d+1}\sum_{j=0}^d \frac{\mathrm{e}^{c_j}}{\prod_{k:~k\ne j}(c_j-c_k)}.
$$
Letting $P:=\Delta_d'$ in Theorem \ref{thm:fourier}, for $c>0$, we obtain
\begin{align*}
\int_{\Delta_d'} \mathrm{e}^{-\bc^\top \bx} d\bx &= \sqrt{d+1}\int_{\R}\frac{\mathrm{e}^{2\pi i y}}{ \prod_{j=0}^{d}(2\pi i y + c_j)}dy~.
\end{align*}
We can use Cauchy's residue theorem to calculate the integral. Take the contour $C$ consisting of the segment $[-R,R]$ $(\pi R>\max{c_j})$ and the upper semicircle $R(\cos \theta+i\sin\theta)$ ($\theta\in[0,\pi]$). Assuming that $c_j\ne c_k$ (for $j\not=k$), we have that the function has the isolated singularities $\frac{ic_j}{2\pi}$. By the residue theorem, we have
\begin{align*}
\int_{C}\frac{\mathrm{e}^{2\pi i z}}{ \prod_{k=0}^{d}(2\pi i z + c_k)} dz &= 2\pi i \sum_{j=0}^d\mathrm{Res}\left(\frac{\mathrm{e}^{2\pi i z}}{ \prod_{k=0}^{d}(2\pi i z + c_k)},\frac{ic_j}{2\pi}\right)\\
&=\sum_{j=0}^d ~ \lim_{z\rightarrow \frac{ic_j}{2\pi}}\frac{\mathrm{e}^{2\pi i z}}{ \prod_{k=0}^{d}(2\pi i z + c_k)}(2\pi i z+c_j)\\
&=\sum_{j=0}^d \frac{\mathrm{e}^{-c_j} }{\prod_{k:k\ne j}(c_k-c_j)}
\end{align*}
In general, assume that $K\subset\{0,1,\dots,d\}$ is the index set of different singularity $\frac{ic_k}{2\pi}$,  and for $k\in K$, let $m_k:=|\{j\in\{0,1,\dots,d\}:~c_j = c_k\}|$. Then
\begin{align}
\int_{C}\frac{\mathrm{e}^{2\pi i z}}{ \prod_{k=0}^{d}(2\pi i z + c_k)} dz &= 2\pi i \sum_{k\in K}\mathrm{Res}\left(\frac{\mathrm{e}^{2\pi i z}}{ \prod_{j=0}^{d}(2\pi i z + c_j)},\frac{ic_k}{2\pi}\right)\notag\\
&= \sum_{k\in K}\mathrm{Res}\left(\frac{\mathrm{e}^{z}}{ \prod_{j=0}^{d}(z + c_j)},-c_k\right). \label{eqn:fourier_in_general}
\end{align}
Also we have
\begin{equation}\label{eqn:contour}
\int_{C}\frac{\mathrm{e}^{2\pi i z}}{ \prod_{k=0}^{d}(2\pi i z + c_k)} dz = \int_{-R}^R \frac{\mathrm{e}^{2\pi i y}}{ \prod_{k=0}^{d}(2\pi i y + c_k)}dy + \int_{|z|=R}\frac{\mathrm{e}^{2\pi i z}}{ \prod_{k=0}^{d}(2\pi i z + c_k)} dz.
\end{equation}
For $|z|=R$, we have $|\mathrm{e}^{2\pi iz}|= |\mathrm{e}^{2\pi i R\cos\theta -2\pi R\sin\theta}|=\mathrm{e}^{-2\pi R\sin\theta}\le 1$. Then we have
$$
\left|\int_{|z|=R}\frac{\mathrm{e}^{2\pi i z}}{ \prod_{k=0}^{d}(2\pi i z + c_k)} dz\right|\le \pi R\sup_{z:~|z|=
R} \left|\frac{\mathrm{e}^{2\pi i z}}{ \prod_{k=0}^{d}(2\pi i z + c_k)}\right|\le \frac{\pi R}{(\pi R)^{d+1}}=\frac{1}{(\pi R)^d}.
$$
Therefore, by taking $R\rightarrow +\infty$ in \eqref{eqn:contour}, we obtain
$$
\int_{\R}\frac{\mathrm{e}^{2\pi i y}}{ \prod_{k=0}^{d}(2\pi i y + c_k)}dy=\sum_{j=0}^d  \frac{\mathrm{e}^{-c_j} }{\prod_{k:k\ne j}(c_k-c_j)}.
$$
Thus, for $c<0$, we have
$$
\int_{\Delta_d'} \mathrm{e}^{\bc^\top \bx} d\bx = \sqrt{d+1}\sum_{j=0}^d \frac{\mathrm{e}^{c_j}}{\prod_{k:~k\ne j}(c_j-c_k)}.
$$
For general $\bc$, there exists $M>\max c_j$ such that $\bc-M\mathbf{1}<0$. Thus
$$
\int_{\Delta_d'} \mathrm{e}^{\bc^\top \bx} d\bx = \mathrm{e}^{M} \int_{\Delta_d'} \mathrm{e}^{(\bc-M\mathbf{1})^\top \bx} d\bx = \sqrt{d+1} ~~\sum_{j=0}^d \frac{\mathrm{e}^{c_j}}{\prod_{k:~k\ne j}(c_j-c_k)}.
$$
Then by affine transformation, we have
$$
\bx\in J \quad \Leftrightarrow \quad \bB^{-1}(\bx-\bv_0)\in\Delta_d \quad \Leftrightarrow \quad \bQ\bB^{-1}(\bx-\bv_0) + \gamma\in \Delta'_d,
$$
where $\bB= \left[\bv_1-\bv_0,~\dots~,\bv_d-\bv_0\right]$ from \eqref{eqn:affine_transformation_B}, $\bQ = \begin{bmatrix}I_d\\ -\mathbf{1}_d^\top\end{bmatrix}$, and $\gamma =\begin{bmatrix}\mathbf{0}_d\\1\end{bmatrix} $ from \eqref{eqn:affine_transformation_Q}.

So, for $\tilde{\bc}^\top \bv_j\ne \tilde{\bc}^\top \bv_k$ $(j\ne k)$, we have
\begin{align*}
\textstyle\int_{J} \mathrm{e}^{\tilde{\bc}^\top \bx} d\bx &\textstyle= |\det(\bB)| \int_{\Delta_d} \mathrm{e}^{\tilde{\bc}^\top (\bB\by+\bv_0)} d\by = \frac{|\det(\bB)|}{\sqrt{d+1}} \int_{\Delta'_d} \mathrm{e}^{\tilde{\bc}^\top (\bB\by+\bv_0)} d\by\\
&\textstyle=\frac{d!\vol(J)}{\sqrt{d+1}}\int_{\Delta'_d} \mathrm{e}^{\tilde{\bc}^\top (\bB\by+\bv_0\sum\limits_{j=0}^d y_j)} d\by\\
&=\textstyle\frac{d!\vol(J)}{\sqrt{d+1}}\int_{\Delta'_d}\mathrm{e}^{\sum\limits_{j=0}^d (\tilde{\bc}^\top \bv_j) y_j} d\by\\
&\textstyle=d!\vol(J)\sum_{j=0}^d \frac{\mathrm{e}^{\tilde{\bc}^\top \bv_j}}{\prod_{k:~k\ne j}(\tilde{\bc}^\top (\bv_j-\bv_k))}.
\end{align*}
\qed \end{proof}

Next, we provide a corollary that is useful for Theorem \ref{thm:naive_exp}. It follows from the same proof idea as Corollary \ref{cor:recover} to connect the integration over $\Delta_d$ with the integration over $\Delta_d'$, but the residues in \eqref{eqn:fourier_in_general} are computed differently.
We consider the integration of the exponential function $\mathrm{e}^{\tilde{\bc}^\top \bx}$ that has exactly two distinct values when evaluated at the vertices of the simplex: one for the vertex $v_0$, and another for all the other vertices.
\begin{cor}\label{cor:fourier_2_value}
Suppose that $J:=\conv\{\bv_0,\bv_1,\dots,\bv_d\}\subset \R^d$.
Suppose further that $u:= \tilde{\bc}^\top \bv_0-\tilde{\bc}^\top \bv_j$ is a nonzero constant for $j=1,\dots,d$, then
\begin{equation*}
\int_{J} \mathrm{e}^{\tilde{\bc}^\top \bx} d\bx =  d! \vol(J)\frac{\mathrm{e}^{\tilde{\bc}^\top \bv_0 -u}}{u^d}\left(\mathrm{e}^{u} -\sum_{j=0}^{d-1}\frac{u^j}{j!}\right).
\end{equation*}
\end{cor}
\begin{proof}
We claim that for $\Delta_d'$ and $\bc$ satisfies $c_j = c_0+u$ for some $u\ne 0$ and $j=1,\dots,d$,
$$
\int_{\Delta_d'} \mathrm{e}^{-\bc^\top \bx} d\bx = \sqrt{d+1}\cdot\frac{\mathrm{e}^{-(c_0+u)}}{u^d}\left(\mathrm{e}^{u}- \sum_{j=0}^{d-1}\frac{u^j}{j!}\right).
$$

Using the same proof as for Corollary \ref{cor:recover} to obtain equation \eqref{eqn:fourier_in_general}, we obtain the above claim by computing the residues in equation \eqref{eqn:fourier_in_general}
\begin{align*}
&\frac{1}{\sqrt{d+1}}\int_{\Delta_d'} \mathrm{e}^{-{\bc}^\top \bx} d\bx 
\\
    &\textstyle \qquad= \mathrm{Res}\left(\frac{\mathrm{e}^{z}}{ \prod_{k=0}^{d}(z + {c}_k)},-{c}_0\right)+ ~\mathrm{Res}\left(\frac{\mathrm{e}^{z}}{ \prod_{k=0}^{d}(z + {c}_k)},-({c}_0+u)\right)\\
    & \qquad =\lim_{z\rightarrow -{c}_0} \textstyle \frac{\mathrm{e}^{z}}{ (z + {c}_0+u)^d} +\frac{1}{(d-1)!}~\displaystyle\lim_{z\rightarrow -({c}_0+u)}\textstyle~\frac{d^{d-1}}{dz^{d-1}}\left(\frac{\mathrm{e}^{z}}{z + {c}_0}\right)\\
    &\textstyle \qquad =\frac{\mathrm{e}^{-{c}_0} }{u^d} -\frac{\mathrm{e}^{-({c}_0+u)}}{u^d}\sum_{j=0}^{d-1}\frac{u^j}{j!} = \frac{\mathrm{e}^{-({c}_0+u)}}{u^d}\left(\mathrm{e}^{u}- \sum_{j=0}^{d-1}\frac{u^j}{j!}\right).
\end{align*}
%
%
So, for $\tilde{\bc}^\top \bv_0 - \tilde{\bc}^\top \bv_j = u$ (i.e., $\tilde{\bc}^\top \bB = -u\mathbf{1}^\top$), applying the affine transformations \eqref{eqn:affine_transformation_B} and \eqref{eqn:affine_transformation_Q}, we have
\begin{align*}
\int_{J} \mathrm{e}^{\tilde{\bc}^\top \bx} d\bx &\textstyle = |\det(\bB)| \int_{\Delta_d} \mathrm{e}^{\tilde{\bc}^\top (\bB\by+\bv_0)} d\by = \frac{|\det(\bB)|}{\sqrt{d+1}} \int_{\Delta'_d} \mathrm{e}^{ (-u\cdot(\mathbf{1}^\top \by)+ \tilde{\bc}^\top \bv_0)} d\by\\
&\textstyle=\frac{d!\vol(J)}{\sqrt{d+1}}\int_{\Delta'_d}\mathrm{e}^{-\bc^\top \by} d\by \quad~(\text{let}~\bc~\text{satisfy}~c_j -c_0=u, {c}_0=-\tilde{\bc}^\top \bv_0)\\
&=\textstyle d!\vol(J)\frac{\mathrm{e}^{\tilde{\bc}^\top \bv_0 -u}}{u^d}\left(\mathrm{e}^{u} -\sum_{j=0}^{d-1}\frac{u^j}{j!}\right).
\end{align*}
The last equation follows from the above claim.
\qed \end{proof}

\subsection{Cubature-rule formulae}\label{subsec:cubature}
In this subsection, we survey the calculation of multidimensional integrals over a simplex via approximate integration formulae, a.k.a., cubature rules (see \cite{stroud1971approximate,cools1993monomial}, \cite{grundmann1978invariant}). 
These formulae are of the form:
\begin{equation}\label{eqn:general_approx_formula}
\int_{\mathcal{R}_d} f(\bx) d\bx = \sum_{j=1}^M \lambda_j f(\bw_j) + Rf,
\end{equation}
where $\mathcal{R}_d$ is a given region in $\R^d$,
the points $\bw_j\in\mathbb{R}^d$, the coefficients $\lambda_j\in\mathbb{R}$ are given, and $Rf$ is the approximation error.
We call $\sum_{j=1}^M\lambda_j f(\bw_j)$ in \eqref{eqn:general_approx_formula} an \emph{integration formula of degree $q$}  if the approximation error $Rf = 0$ for all polynomials $f:~\R^d\rightarrow \R$ of degree at most $q$.
For the univariate case, i.e., $d=1$, the theory of approximate integration is well-established (these formulae in one-dimension are also referred to as quadrature formulae, see \cite{davis2007methods}).

For integration of polynomials of degree at most $q$, we can leverage an approximate integration formula of degree $q$ over a simplex.
Because affine transformation does not change the degree of the polynomial, we focus on integration over the standard simplex $\Delta_d$ in the following and present two formulae for general $d$.
\cite{grundmann1978invariant} gives an invariant integration formula under all affine transformations of $\Delta_d$ onto itself, i.e., under the mappings $\varphi_{\sigma}:(x_1,\dots,x_d)\rightarrow(x_{\sigma_1},\dots,x_{\sigma_d})$, where $\sigma=(\sigma_0,\sigma_1,\dots,\sigma_d)$ is a permutation of $(0,1,\dots,n)$ and $x_0=1-\sum_{j=1}^d x_j$.
They use a combinatorial identity and consider the basis
\[
\textstyle
\left\{(1-\sum_{j=1}^d x_j)^{\alpha_0}x_1^{\alpha_1}\dots x_d^{\alpha_d} :~\|\mathbf{\alpha}\|_1=q,\mathbf{\alpha}\in\Z_{\ge0}^{d+1}\right\},
\]
instead of
the standard monomial basis 
\[
\left\{x_1^{\alpha_1}\dots x_d^{\alpha_d} :~\|\mathbf{\alpha}\|_1\le q,\mathbf{\alpha}\in\Z_{\ge0}^{d}\right\}
\]
for polynomials of degree at most $q$.
%
%
%
\begin{theorem}[{\cite[Thm. 4]{grundmann1978invariant}}]\label{Grundmann_Moller}
Let $q=2s+1$, $s\in\Z_{\ge0}$. Then
\begin{align}
 &\sum_{j=0}^{s} (-1)^j 2^{-2s} \frac{(q+d-2j)^q}{j!(q+d-j)!}\sum_{\|\bk\|_1=s-j,\atop ~\bk\in\Z_{\ge0}^{d+1}} f\left(\frac{2k_1+1}{q+d-2j},\dots,\frac{2k_d+1}{q+d-2j}\right) \notag\\
 = &\int_{\Delta_d} f(\bx) d\bx - Rf \label{eqn:GM}
\end{align}
is an integration formula of degree $q$.
\end{theorem}
This formula is invariant under $\varphi_{\sigma}$ because
\begin{align*}
&\sum_{\|\bk\|_1=s-j,\atop ~\bk\in\Z_{\ge0}^{d+1}} f\left(\frac{2k_1+1}{q+d-2j},\dots,\frac{2k_d+1}{q+d-2j}\right) = \sum_{\|\bk\|_1=s-j,\atop ~k_0\ge k_1\ge \dots\ge k_d} f\left(\left\{\frac{2\mathbf{k}+1}{q+d-2j}\right\}\right),
\end{align*}
where for any point $\by=(y_1,\dots,y_d)\in\R^d$, we define $\{(1-\sum_{j=1}^dy_j, \by)\}$ as the image of all points which are images of $y$ under the mappings $\varphi_{\sigma}$, and we denote $f(\{\by\}): = \sum_{\bw\in\{\by\}} f(\bw)$.

For example, in the case $s=0$, \eqref{eqn:GM} reduces to  the formula ``$T_d:~\mbox{1-1}$'' of  \cite[p. 307]{stroud1971approximate}:
$$\textstyle
\frac{1}{d!}f\left(\frac{1}{d+1},\dots,\frac{1}{d+1}\right) = \int_{\Delta_d} f(\bx) d\bx - Rf.
$$
In the case $s=1$, \eqref{eqn:GM} reduces to the formula ``$T_d:~\mbox{3-1}$'' of \cite[p. 308]{stroud1971approximate}:
\begin{align*}
    &\textstyle\frac{(d+3)^3}{4(d+3)!}f\left(\left\{\frac{3}{d+3},\frac{1}{d+3},\dots,\frac{1}{d+3}\right\}\right)-\frac{(d+1)^3}{4(d+2)!}f\left(\frac{1}{d+1},\dots,\frac{1}{d+1}\right)= \int_{\Delta_d} f(\bx) d\bx - Rf.
\end{align*}
%
%
In general, \eqref{eqn:GM} in Theorem \ref{Grundmann_Moller} requires the evaluation of $f$ at $\sum_{j=0}^s \binom{s-j+d}{d} = \binom{s+d+1}{s}$ points.

By a composition of approximate integration formulae in one dimension, \cite{stroud1971approximate} gives another formula, called the ``conical product formula''.
%
\begin{theorem}[{\cite[pp.  28--31]{stroud1971approximate}}]\label{Str}
There exist  $d$ many approximate integration formulae in one dimension of degree $2s+1$:
\begin{equation}\label{eqn:Gauss_Jacobi}
\int_{0}^1 (1-y_k)^{d-k} f(y_k) d y_k =\sum_{j=1}^{s+1} \lambda_{k,j} f(w_{k,j})+Rf,~\text{for}~k=1,\dots,d.
\end{equation}
Therefore, we can obtain the conical product formula of degree $2s+1$ for $\int_{\Delta_d} f(\bx) d\bx$ with the evaluation of $f$ at $(s+1)^d$  points $\bw_{j_1,j_2,\dots,j_d}=(\nu_{j_1},\nu_{j_1j_2},\dots,\linebreak[1]\nu_{j_1j_2\dots j_d})$ and the corresponding coefficients $\lambda_{j_1,j_2,\dots,j_d}=\lambda_{1,j_1}\lambda_{2,j_2}\dots\lambda_{d,j_d}$, where
\begin{align*}
\nu_{j_1j_2\dots j_k} =& (1-w_{1,j_1})(1-w_{2,j_2})\dots(1-w_{k-1,j_{k-1}}) w_{k,j_k}, \\
&\quad \mbox{ for } k=1,\dots,d,~ 1\le j_k\le s+1.
\end{align*}
\end{theorem}
\begin{proof}
Recall the Gauss-Jacobi quadrature formula of degree $2s+1$ (see \cite{hale2013fast,gil2019noniterative})
$$
\int_{-1}^1 (1-x)^{\alpha}(1+x)^{\beta}f(x) dx = \sum_{j=1}^{s+1} \lambda_j f(w_j) + Rf,
$$
where $w_1,\dots,w_s,w_{s+1}$ are the roots of the Jacobi polynomial 
$$
P^{(\alpha,\beta)}_{s+1}(x):=\frac{(-1)^{s+1}}{2^{s+1} (s+1)!(1-x)^{\alpha}(1+x)^{\beta}}\frac{d^{s+1}}{dx^{s+1}}\left((1-x)^{\alpha}(1+x)^{\beta}(1-x^2)^{s+1}\right),
$$
and 
\begin{align*}
    \lambda_j:=&\int_{-1}^1(1-x)^{\alpha}(1+x)^{\beta}\prod_{k\ne j}\frac{x-w_k}{w_j-w_k} dx\\
    =&\frac{\Gamma(s+1+\alpha+1)\Gamma(s+1+\beta+1)}{\Gamma(s+1+\alpha+\beta+1)(s+1)!}\frac{2^{\alpha+\beta+1}}{(1-w_j^2)[{P^{(\alpha,\beta)}_{s+1}}'(w_j)]^2},
\end{align*}
where  ${P^{(\alpha,\beta)}_{s+1}}'(w_j)$ is the derivative of $P^{(\alpha,\beta)}_{s+1}(x)$ (with respect to $x$) evaluated at $w_j$~.
%

For an arbitrary interval $[a,b]$,
\begin{align*}
&\int_{a}^b(b-x)^{\alpha}(x-a)^{\beta} f(x) dx \\
&\textstyle \qquad = \int_{-1}^1 (1-t)^{\alpha}(1+t)^{\beta} f\left(\frac{(b-a)t+(a+b)}{2}\right)\left(\frac{b-a}{2}\right)^{\alpha+\beta+1} dt \\
&\textstyle \qquad = \sum_{j=1}^{s+1} \left(\frac{b-a}{2}\right)^{\alpha+\beta+1}\lambda_j f\left(\frac{(b-a)w_j+(a+b)}{2}\right) + Rf.
\end{align*}
Thus, letting $a=0$, $b=1$, $\alpha = d-k$ for $k=1,\dots, d$, $\beta=0$, we obtain \eqref{eqn:Gauss_Jacobi}.

Now we rewrite $\int_{\Delta_d} f(\bx)d\bx$ via iterated univariate integration:
$$
\int_{\Delta_d} f(\bx) d\bx = \int_{0}^1\int_0^{1-x_1}\!\!\!\dots\! \int_0^{1-x_1-\dots-x_{d-1}} f(\bx) ~dx_d \dots dx_2 dx_1.
$$
Then we apply the transformation $x_1=y_1$, $x_2=(1-x_1)y_2$, $\dots$, $x_d = (1-x_1-\dots-x_{d-1})y_d$.
Notice that $1-\sum_{k=1}^j x_k=\prod_{k=1}^j(1-y_k)$ and $x_j=y_j\prod_{k=1}^{j-1}(1-y_k)$. The Jacobian determinant of the transformation is
$$
(1-y_1)^{d-1}(1-y_2)^{d-2}\dots(1-y_{d-1}).
$$
So the integration becomes
$$
\int_{\Delta_d}\! f(\bx) d\bx = \!\int_{0}^1\!\int_0^{1}\!\!\!\dots\! \int_0^{1}  (1-y_1)^{d-1}(1-y_2)^{d-2}\dots(1-y_{d-1})f(\bx)~dy_d \dots dy_2 dy_1.
$$
Notice that $x_1^{\alpha_1}\dots x_d^{\alpha_d} = \prod_{j=1}^d y_j^{\alpha_j}(1-y_j)^{\alpha_{j+1}+\dots+\alpha_d}$. Thus, the degree of the integrand with respect to $y_j$ is $\alpha_j+\dots+\alpha_{d}$, which is at most the degree of $f$.
Applying \eqref{eqn:Gauss_Jacobi} for $k=d,d-1,\dots,1$ sequentially, we can obtain the conical product formula of degree $2s$. 
%
\qed \end{proof}

For example, in the case $d=2$, the conical product formula of degree $2s+1$ requires $(s+1)^2$ points, which is no more than $\binom{s+3}{s}=\frac{(s+3)(s+2)(s+1)}{6}$ employed by \eqref{eqn:GM}.
%
%
In the case $s=0$, the conical product formula reduces to  the formula ``$T_d:~\mbox{1-1}$'' of  \cite[p. 307]{stroud1971approximate}.
In the case $s=1$, following the proof of Theorem \ref{Str}, we obtain a formula different from \eqref{eqn:GM}:
%
\begin{align*}
    &\int_{\Delta_2} f(\bx) d\bx = \sum_{j_1=1}^2\sum_{j_2=1}^2\lambda_{1,j_1}\lambda_{2,j_2} f(w_{1,j_1}, (1-w_{1,j_1})w_{2,j_2}), \mbox{ with} \\
    &\textstyle\lambda_{1,1}:=\frac{9+\sqrt{6}}{36}, ~w_{1,1} := \frac{4-\sqrt{6}}{10},~\lambda_{1,2}:=\frac{9-\sqrt{6}}{36}, ~w_{1,2} := \frac{4+\sqrt{6}}{10},\\
    &\textstyle\lambda_{2,1}:=\frac12, ~w_{2,1} := \frac12\left(1+\sqrt{\frac13}\right),~\lambda_{2,2}:=\frac12, ~w_{2,2} := \frac12\left(1-\sqrt{\frac13}\right).
\end{align*}
%
For larger $s$, the exact expressions for the points and weights are tedious, but we can work them out numerically. 
For example, \cite[p. 314]{stroud1971approximate} has ``$T_2:~\mbox{7-1}$'' for $d=2,s=3$ ($16$ points) 
and \cite[p. 315]{stroud1971approximate} has ``$T_3:~\mbox{7-1}$'' for  $d=3,s=3$ ($64$ points).
When $d$ is large, the conical product formula of degree $2s+1$ requires $(s+1)^d$  points, which is more than $\binom{s+d+1}{s}$ points in \eqref{eqn:GM}.

The minimum number of points required in an integration formula of degree $q$ is still open for general $q$. There is a lower bound $\binom{s+d}{s}$ \cite[p. 118--120]{stroud1971approximate} for a formula of degree $2s$. There is a small improvement for the odd degree case $2s+1$. \cite{grundmann1978invariant} conjectures that Theorem \ref{Grundmann_Moller} has the minimum number of points for a formula of degree $2s+1$ when $d\ge 2s$. 
The coefficients in the conical product formula are always positive and sum to $1/d!$, which is a desirable property for numerical stability.
As mentioned in \cite{grundmann1978invariant}, a significant fraction of the
weights in \eqref{eqn:GM} are negative, which might amplify the roundoff errors in the approximation.

Numerically, there are also adaptive algorithms based on a cubature rule and a subdivision strategy to automatically achieve the desired precision of an integral over a general simplex; see \cite{genz2003adaptive}.

%
The cubature formulae problem is also closely related to the symmetric tensor decomposition and Waring's problem; see  \cite{collowald2015problemes,comonSymmetricTensorsSymmetric2008}.
We can view the construction of a cubature formula of degree $q$ as the construction of $\lambda_j$, $w_j$ satisfying
$$
\sum_{j=1}^r \lambda_j p_k(\bw_j) = \int_{\Delta_d} p_k(\bx) d\bx:=\Lambda(p_k), ~k=1,\dots,\textstyle\binom{d+q}{q},
$$
where the $p_k$ form a basis of the polynomials with degree at most $q$.
%
%

A tensor $T\in \R^{(d+1)\times\dots\times(d+1)}$ is called {\emph symmetric} if $t_{j_{\sigma(1)}\dots j_{\sigma(q)}} = t_{j_1\dots j_q}$ for all permutations $\sigma$ on $\{1,\dots, q\}$.
We construct the following tensor $T=[t_{j_1\dots j_q}]\in \R^{(d+1)\times\dots\times(d+1)}$, $j_k\in\{0,1,\dots,d\}$ such that
$$
t_{j_1\dots j_q} = \int_{\Delta_d} x_{j_1}\dots x_{j_q} d\bx,~\text{where}, x_0 = 1-\sum_{j=1}^d x_j = 1-\|\bx\|_1~.
$$
It is easy to see that $T$ is symmetric because 
$$
t_{j_1\dots j_q} = \int_{\Delta_d} x_0^{\alpha_0}\dots x_d^{\alpha_d}d\bx, 
$$
where $\alpha_k, k\in\{0,1,\dots,d\}$ is the number of index $k$ appearing in $j_1,\dots,j_q$, which does not change under the permutations $\sigma$. 
And because $\sum_{k=0}^d \alpha_k = q$, a cubature formula of degree $q$ shows that 
$$
t_{j_1\dots j_q} = \sum_{j=1}^r \lambda_j (\bw_j)_{j_1} \dots (\bw_j)_{j_q}~,
$$
which yields a rank-$r$ symmetric tensor decomposition of $T$:
$$
T = \sum_{j=1}^r \lambda_j \underbrace{(1-\|\bw_j\|_1,\bw_j) \otimes (1-\|\bw_j\|_1,\bw_j) \dots \otimes (1-\|\bw_j\|_1,\bw_j)}_{q~\text{times}}.
$$
%

On the other hand, if there exists a rank-$r$ symmetric tensor decomposition of $T$ with $\by_j\ne 0$:
$$
T = \sum_{j=1}^r \lambda_j \underbrace{\by_j \otimes \by_j \dots \otimes \by_j}_{q~\text{times}}.
$$
We can scale $\by_j$ to satisfy $\|\by_j\|_1=1$, i.e.,
$$
T = \sum_{j=1}^r \lambda_j \|\by_j\|_1^q \underbrace{\frac{\by_j}{\|\by_j\|_1} \otimes \frac{\by_j}{\|\by_j\|_1} \dots \otimes \frac{\by_j}{\|\by_j\|_1}}_{q~\text{times}}.
$$
Therefore, the construction of a cubature formula with minimum points is equivalent to the calculation of the {\emph symmetric rank} of the corresponding tensor $T$.

\cite{comonSymmetricTensorsSymmetric2008} points out the equivalence between symmetric tensor and homogeneous polynomials in $\R[x_0,x_1,\dots,x_d]_q$. The associated homogeneous polynomial to the symmetric tensor $T$ is
\begin{align*}
p_T(\bx) =& \sum_{j_1\dots j_q} t_{j_1\dots j_q} x_0^{\alpha_0(j)}x_1^{\alpha_1(j)}\dots x_d^{\alpha_d(j)}\\ 
=& \sum_{\|\alpha\|_1=q}\frac{q!}{\alpha_0!\alpha_1!\dots\alpha_d!} t_{\alpha_0\alpha_1\dots\alpha_d}x_0^{\alpha_0} x_1^{\alpha_1}\dots x_d^{\alpha_d},\nonumber
\end{align*}
where $\alpha_k(j), k\in\{0,1,\dots,d\}$ is the number of index $k$ appearing in $j_1,\dots,j_q$, and $t_{\alpha_0\alpha_1\dots\alpha_d}:=\int_{\Delta_d}x_0^{\alpha_0}\dots x_d^{\alpha_d} d\bx$.

The symmetric tensor decomposition is closely related to secant varities of the Veronese variety if the field is $\mathbb{C}$.
Consider the map from a vector to a $k$-th power of a linear form:
\begin{align*}
    \nu_{n,k} : \mathbb{C}^n &\rightarrow \mathbb{C}[x_1,\dots,x_n]_{k}\\
                \mathbf{w}   &\rightarrow (w_1x_1+\dots+w_n x_n)^k.
\end{align*}
The image $\nu_{n,k}(\mathbb{C}^n)$ is called the \emph{Veronese variety} and is denoted by $\mathcal{V}_{n,k}$.
Recall the equivalence between symmetric tensor and homogeneous polynomials.
We see that a symmetric tensor of rank 1 corresponds to a point on the Veronese variety. A symmetric tensor of rank at most $r$ lies in the linear space spanned by $r$ points of the Veronese variety.
The closure of the union of all linear spaces spanned by $r$ points of the Veronese variety $\mathcal{V}_{n,k}$ is called the $(r-1)$-th \emph{secant variety} of $\mathcal{V}_{n,k}$.

Therefore, the construction of a cubature formula with minimum points is equivalent to decompose the corresponding homogeneous polynomial to a sum of powers of linear form with minimum number of summands, which is known as the \emph{polynomial Waring problem}.

\section{Comparing na\"{i}ve  and perspective relaxations}\label{sec:multi}
In this section, we present some concrete results comparing
volumes of na\"{i}ve  and perspective relaxations. Quantities of interest are
the \emph{cut-off amount}  
$ \vol(\hyperlink{P0fj}{P^0(f,J)} ) -  \vol(\hyperlink{Pfj}{P(f,J)} ) $ 
and the  \emph{cut-off ratio}
$\frac{\vol(\hyperlink{P0fj}{P^0(f,J)} ) -  \vol(\hyperlink{Pfj}{P(f,J)} )}{ \vol(\hyperlink{P0fj}{P^0(f,J)} )}$. 
For a family of examples, understanding 
when the cut-off ratio is bounded below by a positive quantity or when it
tends to zero and at what rate, gives us information on when the 
na\"{i}ve is an adequate approximation of the  perspective relaxation.

\subsection{$q$-homogeneous functions}

Suppose that $f(\bx)$ is $q$-homogeneous, i.e., $f(\lambda \bx) =\lambda^q f(\bx)$ for $\lambda\ge 0$ ($\lambda=0$ implies that $f(\mathbf{0}) = 0$). Then, for a general simplex, we can compute the volume of the na\"{i}ve relaxation and 
compare it to the  volume of the persepctive relaxation. 
%
\begin{lemma}\label{lem:naivevol_qhomo}
    Suppose that $f$ is continuous, $q$-homogeneous ($q\ge 1$) and convex on $\conv(J\cup \{\mathbf{0}\})$ where the $d$-simplex $J:=\conv\{\bv_0,\allowbreak \bv_1, \dots,\bv_{d}\} \subset \R^d_{\geq 0}\setminus \{ \mathbf{0}\}$. Then
 \begin{align*}
   \vol(\hyperlink{P0fj}{P^0(f,J)} )
    &=\frac{1}{(d+2)!}\left|\det\begin{bmatrix}
        \bv_0 & \bv_1 & \dots & \bv_d\\
        1   &  1  & \dots & 1
    \end{bmatrix}\right|
    \sum_{j=0}^d f(\bv_j) - \frac{1}{q+d+1}\int_{J} f(\bx) d\bx. 
    \end{align*}  
\end{lemma}
\begin{proof}
    %

    By Theorem \ref{thm:naivevol}, we have
    \begin{align*}
   \vol(\hyperlink{P0fj}{P^0(f,J)} ) =\frac{1}{(d+2)!}\left|\det\begin{bmatrix}
        \bv_0 & \bv_1 & \dots & \bv_d\\
        1   &  1  & \dots & 1
    \end{bmatrix}\right|
    \sum_{j=0}^d f(\bv_j) - \int_0^1 z^d\int_J f(z\bx)d\bx dz.
    \end{align*}
    Because $f(\bx)$ is $q$-homogeneous, we have $f(z\bx) = z^q f(\bx)$, and we obtain
    $$
    \int_0^1 z^d\int_J f(z\bx)d\bx dz = \int_{0\le z\le 1} z^{q+d} dz\int_{\bx\in J} f(\bx) d\bx = \frac{1}{q+d+1}\int_{J} f(\bx) d\bx.
    $$
    The result follows.
\qed \end{proof}

\begin{theorem}\label{thm:naivevol_qhomo}
    Suppose that $f$ is continuous, $q$-homogeneous ($q\ge 1$) and convex on $\conv(J\cup \{\mathbf{0}\})$ where the $d$-simplex $J:=\conv\{\bv_0,\allowbreak \bv_1, \dots,\bv_{d}\} \subset \R^d_{\geq 0}\setminus \{ \mathbf{0}\}$. Then
      $$
 \vol(\hyperlink{P0fj}{P^0(f,J)} ) -  \vol(\hyperlink{Pfj}{P(f,J)} )
    = \frac{q-1}{(q+d+1)(d+2)}\int_{J} f(\bx) d\bx.
    $$
\end{theorem}

\begin{proof}
    Recall from Theorem \ref{thm:perspecvol} that the volume of the perspective relaxation \hyperlink{Pfj}{$P(f,J)$} is
    \begin{align*}
   \vol(\hyperlink{Pfj}{P(f,J)} ) &= \frac{1}{(d+2)!}\left|\det\begin{bmatrix}
        \bv_0 & \bv_1 & \dots & \bv_d\\
        1   &  1  & \dots & 1
    \end{bmatrix}\right|
    \sum_{j=0}^d f(\bv_j) - \frac{1}{d+2}\int_{J} f(\bx)d\bx.
    \end{align*}
    Combining this with Lemma \ref{lem:naivevol_qhomo}, the result follows.
\qed\end{proof}
\begin{remark}
Concerning Theorem \ref{thm:naivevol_qhomo}, as a reality check:
(i) when $q=1$, we obtain $ \vol(\hyperlink{P0fj}{P^0(f,J)} ) =  \vol(\hyperlink{Pfj}{P(f,J)} )$, which agrees with the fact that both of these volumes are zero; 
(ii) when $q>1$, because $zf(\bx)\ge f(z\bx) = z^q f(\bx)$ for any $z\in[0,1]$, we have that $f$ is nonnegative,
which implies from the theorem that
$ \vol(\hyperlink{P0fj}{P^0(f,J)} ) \geq  \vol(\hyperlink{Pfj}{P(f,J)} )$ 
(which we know anyway because $\hyperlink{Pfj}{P(f,J)} \subset \hyperlink{P0fj}{P^0(f,J)}$). 
\end{remark}

Theorem \ref{thm:naivevol_qhomo} is a broad generalization of  
the following key result of \cite{lee_gaining_2020} 
giving an expression for the difference of
volumes for convex power functions. 

\begin{cor}[{\cite[Cor. 14]{lee_gaining_2020}}] 
For $d:=1$, $J:=[\ell,u]$ ($u>\ell>0$), and $f(x) := x^q$, with $q>1$, we have 
\[
 \vol(\hyperlink{P0fj}{P^0(f,J)} ) -  \vol(\hyperlink{Pfj}{P(f,J)} ) = \frac{(q-1)(u^{q+1}-\ell^{q+1})}{3(q+2)(q+1)}.
\]
\end{cor}

For  the remainder of Section 4.1, as compared to the hypotheses of Theorem \ref{thm:naivevol_qhomo}, we restrict our attention to 
$f(\bx) := (\bc^\top \bx)^q$ ($\bc\ne \mathbf{0}$) satisfying either 
\begin{enumerate}
\item[(i)] $q>1$ ($q\in\mathbb{R}$) and $\bc^\top \bv_j\ge0$, or
\item[(ii)] $q$ is an even integer (without the assumption $\bc^\top \bv_j\ge0$).
\end{enumerate}
Note that (i) and (ii) each ensure that $f(\bx)$ is convex on $J$. 

We establish in these cases that the cut-off ratio 
has a positive lower bound.  This 
demonstrates that in these cases, the excess volume of
the  na\"{i}ve  relaxation, as compared to the perspective relaxation, is substantial.

\subsubsection{Case (i)}~$f(\bx) = (\bc^\top \bx)^q$ ($q>1$ with further conditions on $\bc$)
\bigskip

Suppose that $q>1$, and $\bc^\top \bv_j\ge 0$ for $j=0,1,\dots,d$.

\begin{lemma}\label{thm:q>1}
For $J:=\conv\{\bv_0,\allowbreak \bv_1, \dots,\bv_{d}\}$, if $\bc\ne\mathbf{0}$, $q\ge 1$ and $\bc^\top \bv_j\ge 0$, then
$$
\int_{J} (\bc^\top \bx)^q d\bx \ge d!\vol(J)\frac{\Gamma(q+1)}{\Gamma(q+d+1)} \sum_{j=0}^d (\bc^\top \bv_j)^q.
$$
The inequality becomes tight when $\frac{\bc^\top \bv_j}{\bc^\top \bv_k}\rightarrow 0$ for all $j\ne k$, where $\bc^\top \bv_k= \max_j{\bc^\top \bv_j}$. 
\end{lemma}
\begin{proof}
$\bx\in J \Leftrightarrow \by= B^{-1}(\bx-\bv_0)\in \Delta_d$, where $B: = \left[\bv_1-\bv_0,~\dots~,\bv_d-\bv_0\right]$.
\begin{align*}
   & \int_{J} (\bc^\top \bx)^q d\bx \\
   &\qquad= d!\vol(J)\int_{\Delta_d} (\bc^\top B \by + \bc^\top \bv_0)^q d\by\\ 
    &\qquad = d!\vol(J)\int_{\Delta_d} [(\bc^\top \bv_1) y_1 + \dots + (\bc^\top \bv_d) y_d + (\bc^\top \bv_0)(1-y_1-\dots-y_d)]^q d\by\\
    &\qquad \ge d!\vol(J)\sum_{j=0}^d \int_{\Delta_d} (\bc^\top \bv_j)^q (y_j)^q d\by, \quad \mbox{ with } y_0 := 1-y_1-\dots-y_d\\
    &\qquad = d!\vol(J)\frac{\Gamma(q+1)}{\Gamma(q+d+1)}\sum_{j=0}^d (\bc^\top \bv_j)^q
\end{align*}
The inequality follows from $(\sum_{j=1}^n x_j)^q\ge \sum_{j=1}^n x_j^q$ when $x_j\ge0$ and $q\ge 1$.
The last equality follows from Proposition \ref{prop:mono_standard_simplex}.
Because $\bc\ne 0$, $\bc^\top \bv_k= \max_j{\bc^\top \bv_j}>0$. The inequality holds tight only when $\frac{\bc^\top \bv_j}{\bc^\top \bv_k}\rightarrow 0$ for all $j\ne k$, where $\bc^\top \bv_k= \max_j{\bc^\top \bv_j}$. 
\qed \end{proof}
\begin{theorem}\label{thm:ratio1}
Suppose that $f(\bx)=(\bc^\top\bx)^q$ ($\bc\ne\mathbf{0}$) with $q> 1$ and $\bc^\top \bv_j\ge0$, $j=0,1,\dots,d$, where the $d$-simplex $J:=\conv\{\bv_0,\allowbreak \bv_1, \dots,\bv_{d}\} \subset \R^d_{\geq 0}\setminus \{ \mathbf{0}\}$. Then
    $$
   \frac{\vol(\hyperlink{P0fj}{P^0(f,J)} ) -  \vol(\hyperlink{Pfj}{P(f,J)} )}{ \vol(\hyperlink{P0fj}{P^0(f,J)} )}
    \ge \frac{(q-1)}{\frac{\Gamma(q+d+2)}{(d+1)!\Gamma(q+1)}-(d+2)}>0.
    $$
The lower bound becomes tight when $\frac{\bc^\top \bv_j}{\bc^\top \bv_k}\rightarrow 0$ for all $j\ne k$, where $\bc^\top \bv_k= \max_j{\bc^\top \bv_j}$.
\end{theorem}
\begin{proof}
By Theorem \ref{thm:naivevol}, Theorem \ref{thm:naivevol_qhomo} and Lemma \ref{thm:q>1},
\begin{align*}
\frac{\vol(\hyperlink{P0fj}{P^0(f,J)} ) -  \vol(\hyperlink{Pfj}{P(f,J)} )}{ \vol(\hyperlink{P0fj}{P^0(f,J)} )}
    & =\frac{q-1}{d+2}\cdot\frac{\int_{J}f(\bx)d\bx}{\frac{(q+d+1)\vol(J)}{(d+2)(d+1)}\sum_{j=0}^d f(\bv_j)-\int_{J}f(\bx)d\bx}\\ &\ge\frac{(q-1)}{\frac{\Gamma(q+d+2)}{(d+1)!\Gamma(q+1)}-(d+2)}>0.
\end{align*}
\qed \end{proof}
For example, when the simplex is parametrized by $u>0$, with $\bv_0$ fixed and $\bv_j =\bv_0+u\be_j$, for $j=1,\dots,d$, and $\bc = \lambda^{1/q}\be_k$ for any $\lambda>0$ and $k=1,\dots,d$, i.e., $f(\bx) = \lambda x_k^q$, we have $\bc^\top \bv_k= \max_j{\bc^\top \bv_j}$ and $\lim_{u\rightarrow +\infty}\frac{\bc^\top \bv_j}{\bc^\top \bv_k}=\lim_{u\rightarrow +\infty}\frac{\bc^\top \bv_0}{\bc^\top\bv_0+u}= 0$ for all $j\ne k$. In this example, the lower bound is asymptotically tight.

For fixed $d$, the lower bound in Theorem \ref{thm:ratio1} has the order of $O(\frac{1}{q^d})$.
Theorem \ref{thm:ratio1} recovers the lower bound on the cut-off ratio from \cite{lee_gaining_2020}. 

\begin{cor}[{\cite[Cor. 17]{lee_gaining_2020}}]
For $d=1$, $J=[\ell,u]$ ($u>\ell>0$), and $f(x)=x^q$, with $q>1$, we have
\[
\frac{\vol(\hyperlink{P0fj}{P^0(f,J)} ) -  \vol(\hyperlink{Pfj}{P(f,J)} )}{ \vol(\hyperlink{P0fj}{P^0(f,J)} )}
 \ge \frac{2}{q+4}.
\]
The lower bound becomes tight only as $\ell/u \rightarrow 0$.
\end{cor}
%


\subsubsection{Case (ii)}~ $f(\bx) = (\bc^\top \bx)^q$ ($q$ is an even integer)
\bigskip 

Suppose that $q$ is an even integer. Similarly, we would prove a lower bound on the ratio between $\int_{J} (\bc^\top \bx)^q d\bx$ and $\sum_{j=0}^d (\bc^\top \bv_j)^q$, aiming at providing a lower bound on the cut-off ratio. 

By Lemma \ref{lem:simplex},
\begin{align*}
    \int_{J} (\bc^\top \bx)^q d\bx = d!\vol(J)\frac{q!}{(q+d)!}h_q(\bc^\top \bv_0,\bc^\top \bv_1,\dots,\bc^\top \bv_d),
\end{align*}
where $h_q(x_1,\dots,x_d) := \sum_{\|\mathbf{k}\|_1=q}x_1^{k_1}\dots x_d^{k_d}$, which is called a complete homogeneous symmetric polynomial \cite{hunter1977positive}. Note that when $J=\Delta_d$, we have 
$$\int_{\Delta_d}(\bc^\top \bx)^q d\bx=\frac{q!}{(q+d)!}h_q(0,c_1,\dots,c_d)=\frac{q!}{(q+d)!}h_q(c_1,\dots,c_d).$$

There are interests in proving even-degree complete homogeneous symmetric polynomials are positive definite (i.e., $h_q(\bx)\ge0$ for all $\bx$, and $h_q(\bx)=0$ if and only if $\bx=\mathbf{0}$)  via different techniques like generating functions, Schur convexity and divided differences \cite{hunter1977positive,tao2017Schur,rovencta2019note}. The integration formula
$$
h_q(c_1,\dots,c_d) = \frac{(q+d)!}{q!}\int_{\Delta_d}(\bc^\top \bx)^q d\bx
$$
would give a simple proof of positive definiteness of $h_q(c_1,\dots,c_d)$ when $q$ is even. This is related to the probability interpretation using i.i.d exponentially distributed random variables mentioned in the comments of the blog \cite{tao2017Schur}:
$$
h_q(c_1,\dots,c_d) = \frac{1}{q!}\int_{\R_{\ge0}^d}(\bc^\top \bx)^q \mathrm{e}^{-\mathbf{1}^\top \bx} d\bx.
$$
The two formulas are connected via a simplification of the multidimensional Laplace form of $f$ (see \cite[Thm. 2.1]{lasserre_simple_2020}).
\begin{lemma}\label{lem:ratio2} 
For $J:=\conv\{\bv_0,\allowbreak \bv_1, \dots,\bv_{d}\}$, if $\bc\ne\mathbf{0}$ and $q$ is an even integer, then
$$
    \int_{J} (\bc^\top \bx)^q d\bx \ge d!\vol(J)\frac{q!}{(q+d)!} \frac{1}{2^{\frac{q}{2}}\left(\frac{q}{2}\right)!}\sum_{j=0}^d(\bc^\top \bv_j)^q.
$$
\end{lemma}
\begin{proof}
Because $q$ is even, \cite[Thm. 1]{hunter1977positive} gives the bound $h_q(c_1,\dots,c_d)\ge \frac{1}{2^{\frac{q}{2}}\left(\frac{q}{2}\right)!}(\sum_{j=1}^d c_j^2)^{\frac{q}{2}}$.
Then, we have 
$$
h_q(c_1,\dots,c_d)\ge\frac{1}{2^{\frac{q}{2}}\left(\frac{q}{2}\right)!}\left(\sum_{j=1}^d c_j^2\right)^{\frac{q}{2}}\ge \frac{1}{2^{\frac{q}{2}}\left(\frac{q}{2}\right)!}\sum_{j=1}^d c_j^q
$$
because $c_j^2\ge0$. Therefore, we obtain the following lower bound 
\begin{align*}
    \int_{J} (\bc^\top \bx)^q d\bx &= d!\vol(J)\frac{q!}{(q+d)!}h_q(\bc^\top \bv_0,\bc^\top \bv_1,\dots,\bc^\top \bv_d)\\
    &\ge d!\vol(J)\frac{q!}{(q+d)!} \frac{1}{2^{\frac{q}{2}}\left(\frac{q}{2}\right)!}\sum_{j=0}^d(\bc^\top \bv_j)^q.
\end{align*}
\qed\end{proof}
It is still an open question whether this bound is tight or not. The bound of \cite{hunter1977positive} is conjectured to be tight and the second inequality is tight. However, the equality conditions are different.
\begin{theorem}\label{thm:ratio2}
Suppose that $f(\bx)=(\bc^\top\bx)^q$ ($\bc\ne\mathbf{0}$) with $q$ an even integer, where the $d$-simplex $J:=\conv\{\bv_0,\allowbreak \bv_1, \dots,\bv_{d}\} \subset \R^d_{\geq 0}\setminus \{ \mathbf{0}\}$. Then
\[
 \frac{\vol(\hyperlink{P0fj}{P^0(f,J)} ) -  \vol(\hyperlink{Pfj}{P(f,J)} )}{ \vol(\hyperlink{P0fj}{P^0(f,J)} )}
    \ge \frac{q-1}{\frac{(q+d+1)!}{q!(d+1)!}2^{\frac{q}{2}}\left(\frac{q}{2}\right)!-(d+2)}.
\]
\end{theorem}
\begin{proof}
By Theorem \ref{thm:naivevol}, Theorem \ref{thm:naivevol_qhomo}  and Lemma \ref{lem:ratio2},
\begin{align*}
  \frac{\vol(\hyperlink{P0fj}{P^0(f,J)} ) -  \vol(\hyperlink{Pfj}{P(f,J)} )}{ \vol(\hyperlink{P0fj}{P^0(f,J)} )}
    & =\frac{q-1}{d+2}\cdot\frac{\int_{J}f(\bx)d\bx}{\frac{(q+d+1)\vol(J)}{(d+2)(d+1)}\sum_{j=0}^d f(\bv_j)-\int_{J}f(\bx)d\bx}\\
    & \ge \frac{q-1}{\frac{(q+d+1)!}{q!(d+1)!}2^{\frac{q}{2}}\left(\frac{q}{2}\right)!-(d+2)}.
\end{align*}
\qed\end{proof}

 For fixed $d$, the lower bound in Theorem \ref{thm:ratio2} has the order of $O\left(\frac{1}{2^{\frac{q}{2}}\left(\frac{q}{2}\right)! q^d}\right)$.
We can improve the coefficient $\gamma$ from $\frac{1}{2^{\frac{q}{2}}\left(\frac{q}{2}\right)!}$ in the inequality $h_q(x_1,\dots,x_d)\ge \gamma \sum_{j=1}^d x_j^q$ for some special cases, and thus improve the lower bound in Theorem \ref{thm:ratio2}.

\begin{proposition}\label{prop:ratio2_improve}
Suppose $h_q(x_1,\dots,x_d) := \sum_{\|\mathbf{k}\|_1=q}x_1^{k_1}\dots x_d^{k_d}$ is the complete homogeneous symmetric polynomial of even degree $q$ with $d$ variables.
In the following three cases (a) $q=2$, or (b) $d=2$, or (c) $d=3$ and $q=4$, we have
$$
    h_q(x_1,\dots,x_d)\ge \frac12 \sum_{j=1}^d x_j^q.
$$
\end{proposition}
\begin{proof}
(a) If $q=2$, then the bound is $h_2(x_1,x_2,\dots,x_d)\ge\frac12\sum_{j=1}^d x_j^q$ and the bound is tight when $\sum_{j=1}^d x_j=0$. The result follows from $h_2(x_1,x_2,\dots,x_d) = \frac12\sum_{j=1}^d x_j^q + \frac12(\sum_{j=1}^d x_j)^2$.

(b) If $d=2$, then the bound is $h_q(x_1,x_2)\ge\frac12(x_1^q+ x_2^q)$ and the bound is tight when $x_1+x_2=0$.
Because $c\ge0$ implies $h_q(x_1,x_2)\ge (x_1^q+ x_2^q)$, we consider the case $x_1>0>x_2$.
\begin{align*}
    \frac{h_q(x_1,x_2)}{x_1^q+x_2^q} = \frac{x_1^{q+1} - x_2^{q+1}}{(x_1-x_2)(x_1^q+x_2^q)}:= R(t) = \frac{1+t^{q+1}}{(1+t)(1+t^q)},
\end{align*}
where $t:=-\frac{x_2}{x_1}>0$. We have
\begin{align*}
 &(1+t)^2 (1+t^q)^2 R'(t)\\
  &\quad = (q+1)t^q(1+t)(1+t^q) -(1+t^{q+1})(1+qt^{q-1}+(q+1)t^q)\\
    &\quad  =t^{2q}+ q t^{q+1} - qt^{q-1} - 1 ~=~ t^q\left(t^q-\frac{1}{t^q} +q (t-\frac{1}{t})\right).
\end{align*}
Because $t^q -\frac{1}{t^q}$, $t-\frac{1}{t}$ is increasing on $(0,\infty)$ and obtain $0$ if and only if $t=1$. Therefore, we know that $R(t)$ is decreasing on $(0,1]$ and increasing on $[1,\infty)$, which implies $\min R(t) = R(1) =\frac{1}{2}$.

(c) If $d=3$ and $q=4$, then we have $h_4(x_1,x_2,x_3)\ge \frac12(x_1^4+x_2^4+x_3^4)$ because
$$
h_4(x_1,x_2,x_3)- \frac12(x_1^4+x_2^4+x_3^4)=\frac12(x_1+x_2+x_3)^2(x_1^2+x_2^2+x_3^2).
$$
\qed\end{proof}
\begin{remark}
The inequality of Proposition \ref{prop:ratio2_improve} does not hold for other cases because
\begin{itemize}
    \item[(i)] $h_6(1,1,-2)/(1^6+1^6+(-2)^6) =31/66\approx 0.4697<1/2$;
    \item[(ii)]  $h_4(0.3577,0.3577,0.3577,-0.9875)/(c_1^4+c_2^4+c_3^4+c_4^4) \approx 0.4598<1/2$.
\end{itemize}
\end{remark}

\subsection{A class of exponential functions}

In this subsection, we present a class of exponential functions and simplices such that the cut-off ratio 
asymptotically goes to 0.
We consider the simplex $J:=\conv\{\bv_0,\bv_1,\dots,\bv_d\}$, and the exponential function $f(\bx) : = \mathrm{e}^{\bc^\top \bx} -1$, where $\bc^\top$ satisfies $\bc^\top \bv_0 - \bc^\top \bv_j = u\ne 0$ for $j=1,\dots,d$.
In other word, the exponential function $f(\bx)=\mathrm{e}^{\bc^\top \bx} -1$ has exactly two distinct values when evaluated at the vertices of the simplex: one for the vertex $v_0$, and another for all the other vertices.
Because this function $f(\bx)$ is not $q$-homogenenous, we need a new theorem (Theorem \ref{thm:naive_exp}) to compute the volume of the perspective and na\"{i}ve relaxations.
\begin{theorem}\label{thm:naive_exp}
Suppose that $J:=\conv\{\bv_0,\bv_1,\dots,\bv_d\}\subset \R^d_{\geq 0}\setminus \{ \mathbf{0}\}$, $f(\bx): = \mathrm{e}^{\bc^\top \bx} -1$ and $\bc$ satisfies $\bc^\top \bv_0 - \bc^\top \bv_j = u\ne 0$, for $j=1,\dots,d$. Then,
\begin{align*}
  \vol(\hyperlink{Pfj}{P(f,J)} ) &=\frac{d!\vol(J)}{(d+2)!}(\mathrm{e}^{\bc^\top \bv_0} + d\mathrm{e}^{\bc^\top \bv_0 - u})-\frac{d!\vol(J)}{d+2} \frac{\mathrm{e}^{\bc^\top \bv_0-u}}{u^d}\left(\mathrm{e}^u - \sum_{j=0}^{d-1}\frac{u^j}{j!}\right),
\end{align*}
and
\begin{align*}
   &\!\! \vol(\hyperlink{P0fj}{P^0(f,J)} )  =\\
   &\quad\frac{d!\vol(J)}{(d+2)!}(\mathrm{e}^{\bc^\top \bv_0} + d\mathrm{e}^{\bc^\top \bv_0 - u}+1)-d!\vol(J)\frac{\mathrm{e}^{\bc^\top \bv_0-u}}{(\bc^\top \bv_0)u^d}\left(\mathrm{e}^u - \sum_{j=0}^{d-1}\frac{u^j}{j!}\right)\\
    &\qquad + d!\vol(J)\frac{\mathrm{e}^{\bc^\top \bv_0-u}}{(\bc^\top \bv_0)[-(\bc^\top \bv_0-u)]^d}\left(\mathrm{e}^{-(\bc^\top \bv_0 - u)}-\sum_{j=0}^{d-1}\frac{[-(\bc^\top \bv_0-u)]^j}{j!}\right).
\end{align*}
\end{theorem}
\begin{proof}
\begin{align*}
&\left|\det\begin{bmatrix}
        \bv_0 & \bv_1 & \dots & \bv_d\\
        1   &  1  & \dots & 1
    \end{bmatrix}\right|
    \sum_{j=0}^d f(\bv_j) ~=~ d!\vol(J)\sum_{j=0}^d f(\bv_j) \\
    &\quad = d!\vol(J) (\mathrm{e}^{\bc^\top \bv_0} + d \mathrm{e}^{\bc^\top \bv_0 - u}- (d+1)).
\end{align*}
By Corollary \ref{cor:fourier_2_value}, we obtain
$$
\int_{J} f(\bx) d\bx = d!\vol(J) \frac{\mathrm{e}^{\bc^\top \bv_0-u}}{u^d}\left(\mathrm{e}^u - \sum_{j=0}^{d-1}\frac{u^j}{j!}\right)- \vol(J).
$$
By Theorem \ref{thm:perspecvol}, we obtain the volume of the perspective relaxation.
By Theorem \ref{thm:naivevol}, we only need to compute $\int_0^1z^d \int_{J}f(z\bx)d\bx dz$ to calculate the volume of the na\"ive relaxation.
We use Corollary \ref{cor:fourier_2_value} and obtain
\begin{align*}
    &\int_0^1z^d \int_{J}f(z\bx)d\bx dz\\
    &\quad= \int_{0}^1 z^d \int_{J} \mathrm{e}^{z\cdot(\bc^\top \bx)} d \bx dz - \frac{\vol(J)}{d+1}\\
    &\quad = d!\vol(J)\int_0^1 z^d  \frac{\mathrm{e}^{z\cdot (\bc^\top \bv_0-u)}}{(zu)^d}\left(\mathrm{e}^{zu} - \sum_{j=0}^{d-1}\frac{{(zu)}^j}{j!}\right) dz- \frac{\vol(J)}{d+1}\\
    &\quad = d!\vol(J)\frac{1}{u^d}\int_0^1 \left(\mathrm{e}^{z\cdot (\bc^\top \bv_0)} - \mathrm{e}^{z\cdot (\bc^\top \bv_0-u)}\sum_{j=0}^{d-1}\frac{{(zu)}^j}{j!}\right) dz- \frac{\vol(J)}{d+1}\\
    &\quad = d!\vol(J)\frac{1}{u^d}\left(\frac{\mathrm{e}^{\bc^\top \bv_0}-1}{\bc^\top \bv_0} -\sum_{j=0}^{d-1}\frac{u^j}{j!}\int_0^1 z^j\mathrm{e}^{z\cdot (\bc^\top \bv_0-u)} dz\right) -\frac{\vol(J)}{d+1}.
\end{align*}
Let $I(j) := \int_0^1 z^{j} \mathrm{e}^{z\cdot (\bc^\top \bv_0-u)}dz$, using  integration by parts, we have
\begin{align*}
I(j+1) &= \frac{1}{\bc^\top \bv_0-u}\int_0^1 z^{j+1} d (\mathrm{e}^{(\bc^\top \bv_0-u)z}) = \left.\frac{1}{\bc^\top \bv_0-u} z^{j+1} \mathrm{e}^{(\bc^\top \bv_0-u)z}\right|_0^1 - \frac{j+1}{\bc^\top \bv_0-u} I(j)\\
&=\frac{\mathrm{e}^{\bc^\top \bv_0-u}}{\bc^\top \bv_0-u} - \frac{j+1}{\bc^\top \bv_0-u} I(j).
\end{align*}
Solving this recursive equation with $I(0)=\frac{1-\mathrm{e}^{\bc^\top \bv_0-u}}{-(\bc^\top \bv_0-u)}$, we obtain
$$
\frac{[-(\bc^\top \bv_0-u)]^{j+1}}{j!}I(j) = 1 - \sum_{k=0}^{j} \frac{[-(\bc^\top \bv_0-u)]^{k}\mathrm{e}^{\bc^\top \bv_0-u}}{k!}.
$$
Thus,
\begin{align*}
    & \quad\frac{\mathrm{e}^{\bc^\top \bv_0}-1}{\bc^\top \bv_0} -\sum_{j=0}^{d-1}\frac{u^j}{j!}\int_0^1 z^j\mathrm{e}^{z\cdot (\bc^\top \bv_0-u)} dz = \frac{\mathrm{e}^{\bc^\top \bv_0}-1}{\bc^\top \bv_0} -\sum_{j=0}^{d-1}\frac{u^j}{j!}I(j)\\
    &= \frac{\mathrm{e}^{\bc^\top \bv_0}-1}{\bc^\top \bv_0} -\sum_{j=0}^{d-1}u^j \frac{1}{[-(\bc^\top \bv_0-u)]^{j+1}}\left(1-\sum_{\ell=0}^{j}\frac{[-(\bc^\top \bv_0-u)]^{\ell}\mathrm{e}^{(\bc^\top \bv_0-u)}}{\ell!}\right)\\
    & = \frac{\mathrm{e}^{\bc^\top \bv_0}-1}{\bc^\top \bv_0} -\sum_{j=0}^{d-1}\frac{u^j}{[-(\bc^\top \bv_0-u)]^{j+1}} + \mathrm{e}^{(\bc^\top \bv_0-u)}\sum_{\ell=0}^{d-1}\frac{[-(\bc^\top \bv_0-u)]^{\ell}}{\ell!}\sum_{j=\ell}^{d-1}\frac{u^j}{[-(\bc^\top \bv_0-u)]^{j+1}}\\
    & = \frac{\mathrm{e}^{\bc^\top \bv_0}-1}{\bc^\top \bv_0} +\frac{1-u^d[-(\bc^\top \bv_0-u)]^{-d}}{\bc^\top \bv_0} - \mathrm{e}^{(\bc^\top \bv_0-u)}\sum_{\ell=0}^{d-1}\frac{u^\ell - u^d[-(\bc^\top \bv_0-u)]^{\ell-d}}{\ell!(\bc^\top\bv_0)}\\
    & =\frac{\mathrm{e}^{\bc^\top \bv_0 - u}}{\bc^\top \bv_0}\left(\mathrm{e}^u-\sum_{j=0}^{d-1}\frac{u^j}{j!}\right)  -\frac{u^d \mathrm{e}^{\bc^\top \bv_0-u}}{(\bc^\top \bv_0)[-(\bc^\top \bv_0-u)]^d}\left(\mathrm{e}^{-(\bc^\top \bv_0 - u)}-\sum_{j=0}^{d-1}\frac{[-(\bc^\top \bv_0-u)]^j}{j!}\right),
\end{align*}
where the second last equality follows from the geometric series. Therefore, we obtain the formula for $ \vol(\hyperlink{P0fj}{P^0(f,J)} ) $.
\qed \end{proof}

Next, we present two families of simplices for the exponential function $f(\bx) = \mathrm{e}^{\mathbf{1}^\top \bx} - 1$ such that the 
cut-off ratio 
asymptotically goes to 0, and we establish the rate of convergence for each by Theorem \ref{thm:naive_exp}.
\begin{theorem}\label{thm:exp_ratio}
(a) Suppose  that $J := \{\bx:~\bx\le ku\mathbf{1},~\|\bx-ku\mathbf{1}\|\le u\}=\conv\{\bv_0,\bv_0-u\be_1,\dots,\bv_0-u\be_d\}$, where $\bv_0:=ku\mathbf{1}\in \R_{>0}^d$, and $f(\bx) := \mathrm{e}^{\mathbf{1}^\top \bx} - 1$. Then,
$$
    \lim_{u\rightarrow \infty}u^d\cdot 
\frac{\vol(\hyperlink{P0fj}{P^0(f,J)} ) -  \vol(\hyperlink{Pfj}{P(f,J)} )}{ \vol(\hyperlink{P0fj}{P^0(f,J)} )}
    = (d+1)!
$$
(b) Suppose $J := \conv\{\bv_0,\bv_0 + u\be_1,\dots,\bv_0+u \be_d\}$, where $\mathbf{0}\neq \bv_0\in\R^d_{\ge0}$, and $f(x) = \mathrm{e}^{\mathbf{1}^\top \bx} -1$. Suppose that $\bv_0$ is fixed and $u$ tends to $\infty$. Then,
$$
    \lim_{u\rightarrow \infty}u\cdot 
\frac{\vol(\hyperlink{P0fj}{P^0(f,J)} ) -  \vol(\hyperlink{Pfj}{P(f,J)} )}{ \vol(\hyperlink{P0fj}{P^0(f,J)} )}
    = d+1.
$$
\end{theorem}
\begin{proof}
(a) Because $\mathbf{1}^\top\bv_0 = kdu$, $\mathbf{1}^\top (\bv_0-u\be_j) = (kd-1)u$, by Theorem \ref{thm:naive_exp}, we collect the highest-order term as $u$ tends to infinity and obtain,
\begin{align*}
 \vol(\hyperlink{P0fj}{P^0(f,J)} ) -  \vol(\hyperlink{Pfj}{P(f,J)} )\sim \frac{d!\vol(J)}{d+2}\frac{\mathrm{e}^{kdu}}{u^d}& \\
 \hyperlink{P0fj}{P^0(f,J)} \sim \frac{d!\vol(J)}{(d+2)!}\mathrm{e}^{kdu}& \\
    \lim_{u\rightarrow \infty}u^d\cdot 
\frac{\vol(\hyperlink{P0fj}{P^0(f,J)} ) -  \vol(\hyperlink{Pfj}{P(f,J)} )}{ \vol(\hyperlink{P0fj}{P^0(f,J)} )}
    &= (d+1)!
\end{align*}
(b) Because $\mathbf{1}^\top (\bv_0+u\be_j) = \mathbf{1}^\top\bv_0 + u$, by Theorem \ref{thm:naive_exp}, we collect the highest-order term as $u$ tends to infinity and obtain,
\begin{align*}
    \vol(\hyperlink{P0fj}{P^0(f,J)} ) -  \vol(\hyperlink{Pfj}{P(f,J)} ) \sim \frac{d!\vol(J)}{d+2}\frac{\mathrm{e}^{\mathbf{1}^\top \bv_0}}{(d-1)!}\frac{\mathrm{e}^u}{u}&\\ 
    \vol(\hyperlink{P0fj}{P^0(f,J)} ) \sim \frac{d!\vol(J)}{(d+2)!}d\mathrm{e}^{\mathbf{1}^\top \bv_0} \mathrm{e}^u &\\
    \lim_{u\rightarrow \infty}u\cdot 
\frac{\vol(\hyperlink{P0fj}{P^0(f,J)} ) -  \vol(\hyperlink{Pfj}{P(f,J)} )}{ \vol(\hyperlink{P0fj}{P^0(f,J)} )}
    &= d+1.
\end{align*}
\qed\end{proof}
Theorem \ref{thm:exp_ratio}(a) recovers the following key result of \cite{lee_gaining_2020}. 

\begin{cor}[{\cite[Cor. 6]{lee_gaining_2020}}]
For $d=1$, $J:=[\ell,u]$ ($u>\ell>0$), and $f(x):=e^x-1$. Let $\ell:=k u$ for some fixed $k\in(0,1)$, then we have
\[
 \lim_{u\rightarrow \infty} u\cdot 
\frac{\vol(\hyperlink{P0fj}{P^0(f,J)} ) -  \vol(\hyperlink{Pfj}{P(f,J)} )}{ \vol(\hyperlink{P0fj}{P^0(f,J)} )}
 =\frac{2}{1-k}.
\]
\end{cor}

\subsection{The log-sum-exp function and more}
For some convex functions, there are no closed-form formulae for integration over a simplex. In such a case, we can numerically approximate the integration using the cubature formulae presented in Section \ref{subsec:cubature} and compute the asymptotic ratio.
Suppose that we have a cubature formula of degree $q$:
\begin{equation*}
\int_{\Delta_d} f(\bx) d\bx = \sum_{j=1}^M \lambda_j f(\bw_j) + Rf,
\end{equation*}
where $Rf$ is the approximation error, and $Rf=0$ for all polynomials $f$ of degree at most $q$. And we can use $\sum_{j=1}^M \lambda_j f(\bw_j)$ with $M$ summands to approximate the integration $\int_{\Delta_d} f(\bx) d\bx $.
Then, after affine transformation, we can approximate $\int_{J} f(\bx) d\bx$ as follows:
\begin{equation*}
\int_{J} f(\bx) d\bx\approx \sum_{j=1}^M |\det \bB|\lambda_j f(\bB\bw_j+\bv_0),
\end{equation*}
where $\bx\in J ~\Leftrightarrow~ \bB^{-1}(\bx-\bv_0)\in \Delta_d$ as in \eqref{eqn:affine_transformation_B}.
Therefore, we can calculate $\vol(P(f,J))$ (see Theorem \ref{thm:perspecvol})
\[
\vol(\hyperlink{Pfj}{P(f,J)} )\approx\frac{|\det \bB|}{(d+2)!}\sum_{j=0}^d f(\bv_j) - \frac{|\det \bB|}{d+2}\sum_{j=1}^M \lambda_j f(\bB\bw_j+\bv_0).
\]
To compute $\vol(\hyperlink{P0fj}{P^0(f,J)} )$, we need a cubature formula for the region $\{(\bx,z):~\bx\in z\cdot J, 0\le z\le 1\}$.
\begin{theorem}[Theorem 2.8-1 in \cite{stroud1971approximate}]
Suppose that we have a cubature formula of degree $q$ for a region $J$
\begin{equation*}
\int_{J} f(\bx) d\bx = \sum_{j=1}^M \lambda_j f(\bw_j) + R_1f,
\end{equation*}
and a cubature formula of degree $q$
\begin{equation*}
\int_0^1 z^d f(z) dz = \sum_{k=1}^{N} \nu_k f(r_k) + R_2f.
\end{equation*}
Then we have a cubature formula of degree $q$
\begin{equation*}
\iint_{\bx\in z\cdot J, 0\le z\le 1} f(\bx) d\bx dz = \sum_{j=1}^{M}\sum_{k=1}^N \lambda_j\nu_k f(r_k\bw_j) + R_3f.
\end{equation*}
\end{theorem}

Therefore, we can calculate $\vol(P^0(f,J))$ (see Theorem \ref{thm:naivevol})
$$
 \vol(\hyperlink{P0fj}{P^0(f,J)} )\approx\frac{|\det \bB|}{(d+2)!}\sum_{j=0}^d f(\bv_j) - |\det \bB| \sum_{j=1}^{M}\sum_{k=1}^N \lambda_j\nu_k f(r_k(\bB \bw_j+\bv_0)).
$$

Next, we test on an example with $d=3$, the log-sum-exp function $f(\bx) := \log \frac{e^{x_1}+e^{x_2}+e^{x_3}}{3}$ (see Example \ref{ex:logsumexp}), and the scaled standard simplex $J:=u\cdot \Delta_d$.
We use the cubature formula of degree 5 in Theroem \ref{Str} for $\Delta_d$ (\url{https://www.mathworks.com/matlabcentral/fileexchange/9435-n-dimensional-simplex-quadrature}) and the Gauss-Jacobi quadrature formula of degree 5 for $\int_0^1 z^d f(z) dz$ (\url{https://www.mathworks.com/matlabcentral/fileexchange/65674-gauss-jacobi-quadrature-rule-n-a-b}) to numerically approximate  the cut-off ratio. We have
%
\begin{align*}
\hyperlink{Pfj}{P(f,J)} &\approx\frac{u^d}{(d+2)!}\sum_{j=1}^d f(u\be_j) - \frac{u^d}{d+2}\sum_{j=1}^M \lambda_j f(u\bw_j),\\
\hyperlink{P0fj}{P^0(f,J)}&\approx\frac{u^d}{(d+2)!}\sum_{j=1}^d f(u\be_j) - u^d \sum_{j=1}^{M}\sum_{k=1}^N \lambda_j\nu_k f(u r_k \bw_j),
\end{align*}
where $M$ and $N$ are the number of summands in the cubature formula of degree 5 in Theroem \ref{Str} for $\Delta_d$ and the Gauss-Jacobi quadrature formula of degree 5 for $\int_0^1 z^d f(z) dz$, respectively.
Figure \ref{fig:numer} shows that the approximated cut-off ratio is 
always small and quickly tends to 
decrease, thus demonstrating that for this family of examples, the 
na\"{i}ve relaxation is quite good. 
\begin{figure}[H]
\centering
\includegraphics[width=.7\textwidth]{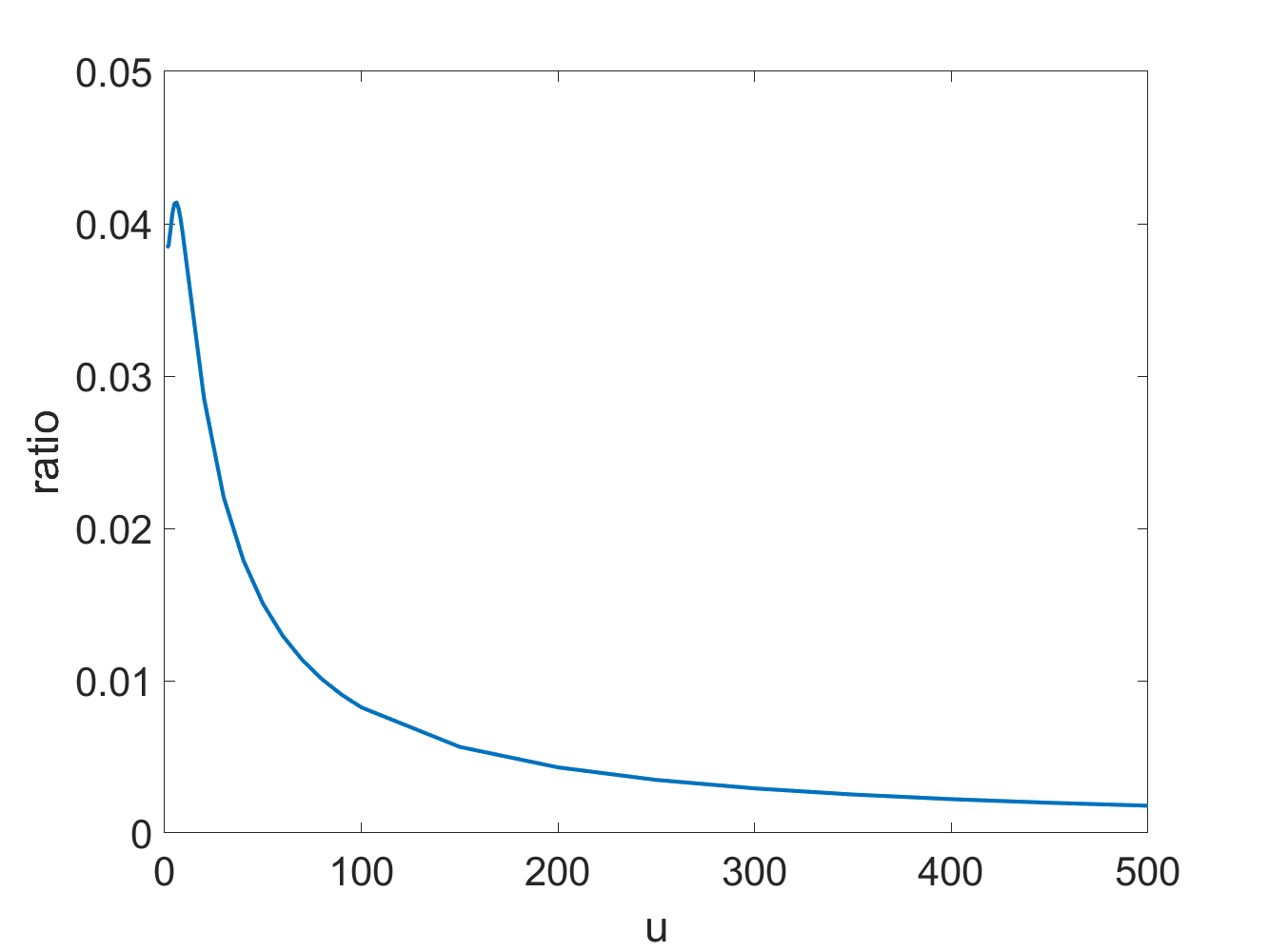}
\caption{The approximated cut-off ratio for a log-sum-exp function $f(\bx) = \log \frac{e^{x_1}+e^{x_2}+e^{x_3}}{3}$ with respect to an expanding simplex $J=u\cdot \Delta_d$}\label{fig:numer}
\end{figure}
In fact, in what follows, we prove this apparent limiting behavior (for arbitrary $d$
and even when the base of the simplex is shifted), and at the same time providing 
some validation of the approximation that we made above using cubature.

\begin{lemma}\label{lem:log_sum_exp}
Suppose that $v_j\in\mathbb{R}^d$, for  $j=1,\dots,d$. Then
$$
\lim_{u\rightarrow \infty}\textstyle \frac{1}{u}\int_{\Delta_d} \log \left(\frac{1}{d}\sum_{j=1}^d e^{u x_j+v_j}\right) d\bx= \int_{\Delta_d}\max(\bx) d\bx=\frac{1}{(d+1)!}\sum_{j=1}^d\frac{1}{j}.
$$
\end{lemma}
\begin{proof}
Notice that
$$
\textstyle
\log \left(\frac1d e^{u \max(\bx)+\min(\bv)}\right)\le \log \left(\frac{1}{d}\sum_{j=1}^d e^{u x_j+v_j}\right) \le \log \left(\frac1d e^{u \max(\bx)}\sum_{j=1}^d e^{v_j}\right).
$$
We have 
$$
\textstyle
\max(\bx) + \frac{\min(\bv)-\log d}{u}\le \frac{1}{u}\log \left(\frac{1}{d}\sum_{j=1}^d e^{u x_j+v_j}\right)\le \max(\bx)+\frac{\log \left(\sum_{j=1}^d e^{v_j}\right)-\log d}{u}.
$$ 
Therefore, 
$\lim_{u\rightarrow \infty}\frac{1}{u}\int_{\Delta_d} \log \left(\frac{1}{d}\sum_{j=1}^d e^{u x_j+v_j}\right) d\bx= \int_{\Delta_d}\max(\bx) d\bx.$
By \cite[Thm. 2.1]{lasserre_simple_2020}, we have
$$
\textstyle
\int_{\Delta_d} \max(\bx) d\bx = \frac{1}{(d+1)!}\int_{\mathbb{R}_{\ge 0}} \max(\bx) e^{-\mathbf{1}^\top \bx} d\bx =\frac{1}{(d+1)!} \mathbb{E}(\max(X_1,\dots,X_d)),
$$
where $X_1,\dots,X_d$ are i.i.d. exponential random variables with mean 1, and $\mathbb{E}(\cdot)$ denotes the expectation.
By \cite[Eq. 1.9]{renyi1953theory}, the lemma follows.
\qed\end{proof}
\begin{theorem}
Let $J :=  \bv_0+ u\Delta_d$,  where $\mathbf{0}\neq \bv_0\in\R^d_{\ge0}$~, and $f(\bx) := \log \left(\frac{1}{d}\sum_{j=1}^d e^{x_j}\right)$. 
Then,
$$
    \lim_{u\rightarrow \infty}\textstyle
\frac{\vol(\hyperlink{P0fj}{P^0(f,J)} ) -  \vol(\hyperlink{Pfj}{P(f,J)} )}{ \vol(\hyperlink{P0fj}{P^0(f,J)} )}
    = 0.
$$
\end{theorem}
\begin{proof}
By Lemma \ref{lem:log_sum_exp}, $\lim\limits_{u\rightarrow \infty}\frac{1}{u}\int_{\Delta_d} f(u\bx+\bv_0)d\bx = C_d$, where $C_d=\frac{\sum_{j=1}^d\frac{1}{j}}{(d+1)!}$. By Theorem \ref{thm:perspecvol} and \ref{thm:naivevol}, we have
\begin{align*}
    &\textstyle
    \frac{ \vol(\hyperlink{P0fj}{P^0(f,J)} ) -  \vol(\hyperlink{Pfj}{P(f,J)} )}{d!\vol(J)} 
  = \textstyle\frac{1}{d+2}\int_{\Delta_d} f(u\bx+\bv_0) d\bx-  \int_0^1 z^d\int_{\Delta_d} f(z(u\bx+\bv_0)) d\bx dz.
\end{align*}
Thus,
\begin{align*}
    \lim_{u\rightarrow\infty}\textstyle \frac{1}{u}\frac{ \vol(\hyperlink{P0fj}{P^0(f,J)} ) -  \vol(\hyperlink{Pfj}{P(f,J)} )}{d!\vol(J)} &=\textstyle\frac{1}{d+2}C_d - \int_0^1z^{d+1} \lim\limits_{u\rightarrow \infty}\frac{1}{uz}\int_{\Delta_d} f(z(u\bx+\bv_0))d\bx dz \\
    &= \frac{1}{d+2}C_d-  \int_0^1 z^{d+1}C_d dz=0. 
\end{align*}
We can further compute
\begin{align*}    
    \textstyle\frac{ \vol(\hyperlink{P0fj}{P^0(f,J)} ) }{d!\vol(J)} &\textstyle= \frac{1}{(d+2)!}\left(f(\bv_0)+\sum_{j=1}^d f(u\be_j+\bv_0)\right)\\
    &\textstyle\qquad - \int_0^1 z^d\int_{\Delta_d} f(z(u\bx+\bv_0)) d\bx dz.
\end{align*}
Thus,
\begin{align*}  
    \lim_{u\rightarrow\infty} \textstyle \frac{1}{u}&\textstyle\frac{ \vol(\hyperlink{P0fj}{P^0(f,J)} ) }{d!\vol(J)}= \frac{d}{(d+2)!} - \frac{1}{d+2}C_d=\frac{1}{(d+2)!}\left(d-\sum_{j=1}^d\frac{1}{j}\right)>0.
\end{align*}
Therefore, the result follows.
\qed\end{proof}

\section{Conclusions}\label{sec:conc}
We investigated the idea of using volume as a measure to compare the perspective relaxation and na\"{i}ve relaxation in the multivariate case, for a natural disjunctive model that received a lot of attention in
the univariate case. Focusing on the natural and  fundamental building-block case where the domain is a simplex, we extended some results for the univariate case.

\begin{itemize}
\item We provided a theorem to compute the volumes of the perspective relaxation and na\"{i}ve relaxation for general functions and connect the calculation to the integration over a simplex. We 
made an extensive survey of the relevant results on integration over a simplex,
working out the connections and some extensions (which might
additionally be of independent interest for the optimization community). 
\item We analyzed the cut-off ratio for several important classes of functions,  generalizing
results from the univariate case.
Specifically, the cut-off ratio has a positive lower bound for powers of linear functions under some conditions, which implies that the difference between the two relaxations is substantial.
On the other side, the cut-off ratio is small and tends to 0 for a class of exponential functions and 
the log-sum-exp function over a scaled standard simplex, which implies that the perspective  and na\"{i}ve relaxations  are close.
\item When the closed formula is not available, we provided an idea on how to use cubature formulae to numerically compute and compare the volumes.
%
%
\end{itemize}

For  future directions, we believe that some technical improvements can be achieved, for example, the lower bound in Theorem \ref{thm:ratio2}, as well as understanding the asymptotic behavior of the cut-off ratio in terms of more general classes of 
functions and domains.
A further interesting direction is to generalize and compare other relaxations for the multivariate setting, such as extending the original function, and perspective relaxation of the piecewise-linear under-estimators
(see \cite{lee2020piecewise}). However, we probably need stronger assumptions on the functions and simplex to handle other relaxation in the multivariate setting.

Finally, we briefly discuss a general setting when the decision variable (vector) $\bx$
is either $\mathbf{0}\in\mathbb{R}^d$ or in a polytope $P\subset \mathbb{R}_{\geq 0}^d$ (not containing $\mathbf{0}$), and we have a triangulation of the convex polytope $P$.
We are considering convex relaxations of the ``disjunctive set''
\begin{align*}
    D(f,\mathcal{J})&:= \left\{\mathbf{0}_{d+1+|\mathcal{N}|}\right\}\cup\\
    &\bigcup_{n\in\mathcal{N}}\left\{(\bx,y,\bz)\in\R^d \times \R\times \{0,1\}^{|\mathcal{N}|}~:~y = f(\bx),~ \bx\in J_{n},~ \bz=\be_n^{|\mathcal{N}|}\right\},
\end{align*}
where $\mathcal{J}=\{J_n:~n\in\mathcal{N}\}$ is a triangulation of the convex polytope domain in $\R^d$, and $f$ is convex on $J_n$, for $n\in\mathcal{N}$. We assume that the polytope domain is a subset of $\mathbb{R}_{\geq 0}^d\setminus \{\mathbf{0}\}$.
The binary $|\mathcal{N}|$-vector $\bz$ is either $\mathbf{0}$,
if $(\bx,y)=(\mathbf{0},0)$ or $\bz$ is the $n$-th standard unit vector, if $\bx\in J_n$ for some $n\in\mathcal{N}$.
The special case with $|\mathcal{N}|=1$ is what we analyzed in this work. 
In applications of the general case, $D(f,\mathcal{J})$ would be a substructure of a larger model, where 
the cost of $\bx\in J_n$ is $f(\bx) + c_n$~, and is modeled by $y+\sum_{n\in\mathcal{N}} c_n z_n$~.

%
Let $\mu_n(\bx)$ be a linear function that bounds $y$ from above on $J_n$, $n\in\mathcal{N}$. 
%
By introducing $\bx_n$ and $y_n$ for each simplex $J_n$, $n\in\mathcal{N}$, we obtain the \emph{extended perspective relaxation}
%
\begin{align*}
P(f,\mathcal{J}):=&\mathrm{cl}\left\{ (\{\bx_n:n\in\mathcal{N}\},\by,\bz) \in \R^{d|\mathcal{N}|}\times \R^{|\mathcal{N}|} \times (0,1]^{|\mathcal{N}|} ~:~ \right.\\
&\left.\tilde{\mu}_n(\bx_n,z_n) \geq y_n \geq \tilde{f}(\bx_n,z_n),~
 \mathbf{1}^\top \bz\le 1, \bx_j\in z_n\cdot J_n~,~ n\in \mathcal{N}
\vphantom{\R^{|\mathcal{N}|}}\right\},
\end{align*}
where $\bx=\sum_{n\in\mathcal{N}}\bx_n$~, and $y=\sum_{n\in\mathcal{N}}y_n$~.
It is only the constraint $\mathbf{1}^\top \bz\le 1$ that prevents the extended perspective relaxation from factoring across the set of simplices $\mathcal{J}$, and hopefully we can still use the analysis
for each subproblem on a single simplex. We see going deeper into analyzing $P(f,\mathcal{J})$ as a starting point for some important further work on our topic.

%


\begin{acknowledgements}
We gratefully acknowledge discussions with Mathias K\"oppe in regard to \cite{baldoni_how_2010}.
\end{acknowledgements}

%
\section*{Conflict of interest}
The authors declare that they have no conflict of interest.

\bibliographystyle{alpha}
\bibliography{perssimplex}   


\end{document}